\documentclass[10pt]{amsart}

\title[A generalization of the DR cycle via log-geometry]{A generalization of the double ramification cycle via log-geometry}

\author[Gu\'er\'e]{J\'er\'emy Gu\'er\'e}
\address{Humboldt Universit\"at\\
Berlin\\
Germany}
\email{jeremy.guere@gmail.com}

\usepackage{amsrefs}
\usepackage{amsfonts}
\usepackage{amsmath}
\usepackage{amssymb}
\usepackage{amscd}
\usepackage{array}
\usepackage{subfigure}
\usepackage{tikz}

\usepackage[all]{xy}

\allowdisplaybreaks[3]

\newcommand{\kM}{\mathfrak{M}}
\newcommand{\kC}{\mathfrak{C}}
\newcommand{\kP}{\mathfrak{P}}

\newcommand{\PP}{\mathbb{P}}
\newcommand{\CC}{\mathbb{C}}
\newcommand{\ZZ}{\mathbb{Z}}
\newcommand{\NN}{\mathbb{N}}

\newcommand{\cO}{\mathcal{O}}

\newcommand{\cC}{\mathcal{C}}
\newcommand{\cD}{\mathcal{D}}

\newcommand{\cH}{\mathcal{H}}

\newcommand{\cL}{\mathcal{L}}
\newcommand{\cM}{\mathcal{M}}

\theoremstyle{plain}

\newtheorem{thm}{Theorem}[section]
\newtheorem{pro}[thm]{Proposition}
\newtheorem{lem}[thm]{Lemma}
\newtheorem*{lem*}{Lemma}

\newtheorem{cor}[thm]{Corollary}
\newtheorem{conj}[thm]{Conjecture}

\theoremstyle{definition}

\newtheorem{dfn}[thm]{Definition}

\newtheorem{rem}[thm]{Remark}
\newtheorem*{rem*}{Remark}
\newtheorem*{rems*}{Remarks}

\newtheorem*{exa*}{Example}
\newtheorem*{exait*}{\rm \em Example}
\newtheorem*{exadefit*}{\rm \em Example/Definition}
\newtheorem*{cla*}{\rm \em Claim}

\newtheorem*{dfn*}{Definition}

\usepackage{amsxtra}

\def\<{\left\langle}
\def\>{\right\rangle}

\begin{document}
\begin{abstract}
We give a log-geometric description of the space of twisted canonical divisors constructed by Farkas--Pandharipande. In particular, we introduce the notion of a principal rubber $k$-log-canonical divisor, and we study its moduli space. It is a proper Deligne--Mumford stack admitting a perfect obstruction theory whose virtual fundamental cycle is of dimension $2g-3+n$. In the so-called strictly meromorphic case with $k=1$, the moduli space is of the expected dimension and the push-forward of its virtual fundamental cycle to the moduli space of stable curves equals the weighted fundamental class of the moduli space of twisted canonical divisors. Conjecturally, it yields a formula of Pixton generalizing the double ramification cycle in the moduli space of stable curves.
\end{abstract}

\maketitle

\tableofcontents

\setcounter{section}{-1}

\section{Introduction}
The double ramification cycle is a codimension-$g$ cycle $\mathrm{DR}_g(\vec{\mathbf{m}})$ in the moduli space of stable curves $\overline{\cM}_{g,n}$ attached to integers $\vec{\mathbf{m}} = (\mathbf{m}_1,\dotsc,\mathbf{m}_n)$ whose sum is zero.
On the open part $\cM_{g,n}$ consisting of smooth curves, it is described as
$$\left\lbrace \left[C,\sigma_1,\dotsc,\sigma_n\right] \in \cM_{g,n} ~ | ~ \cO_C(\mathbf{m}_1 \sigma_1 + \dotsb + \mathbf{m}_n \sigma_n) \simeq \cO_C \right\rbrace.$$
The double ramification cycle over the whole moduli space $\overline{\cM}_{g,n}$ is defined using Gromov--Witten theory of rubber $\PP^1$.
Precisely, the moduli space of the theory consists of stable maps from prestable curves to the target $\PP^1$ relative to the divisor $\left[0\right] + \left[\infty\right]$, up to the natural $\CC^*$-action on $\PP^1$, and with ramification profile at $0$ (resp.~at $\infty$) given by the positive (resp.~negative) integers among $\vec{\mathbf{m}}$.
The moduli space of relative stable maps has been developed in the
symplectic setting by Li--Ruan \cite{LR01} and Ionel--Parker \cite{IP03,IP04}, and in the algebraic setting by Li \cite{Li01,Li02}.
The idea is roughly to allow expansions of the target by attaching some extra $\PP^1$ at the divisor.
Li has developed a perfect obstruction theory and eventually the double ramification cycle is the push-forward of the virtual fundamental cycle
$$\mathrm{DR}_g(\vec{\mathbf{m}}) := \mathrm{st}_* \left[ \overline{\cM}_g(\PP^1,\vec{\mathbf{m}})^\sim \right]^\mathrm{vir} \in A_*(\overline{\cM}_{g,n})$$
of the moduli space of relative stable maps to rubber $\PP^1$, where the morphism $\mathrm{st}$ is the functor forgetting the map and the target and stabilizing the curve.

The double ramification cycle has been extensively studied in the context of moduli spaces, and appears now as a central object in the context of integrable hierarchies, see \cite{BurDR,BDGR}.
A precise formula as a sum of tautological classes in the moduli space of stable curves of compact type, i.e.~stable curves whose dual graph is a tree, has been obtained by Hain \cite{Hain}, but a formula for the double ramification cycle in the whole space $\overline{\cM}_{g,n}$ is much more involved.
It has been conjecturally found by Pixton \cite{Pixton} and proved in \cite{Janda}, and is based on an extension of Hain's formula as a sum over all dual graphs and on an ingenious regularization process, described for instance in \cite[Appendix A.3]{PandhFark} and in \cite{Janda}.

Interestingly, both Pixton's formula and the double ramification cycle have natural generalizations indexed by $k \in \NN$, and attached to a partition $\vec{\mathbf{m}}=(\mathbf{m}_1,\dotsc,\mathbf{m}_n)$ of $k(2g-2)$.
For instance, on the open part $\cM_{g,n}$, the generalization of the double ramification cycle is described as
$$\cH^k_g(\vec{\mathbf{m}}) := \left\lbrace \left[C,\sigma_1,\dotsc,\sigma_n\right] \in \cM_{g,n} ~ | ~ \cO_C(\mathbf{m}_1 \sigma_1 + \dotsb + \mathbf{m}_n \sigma_n) \simeq \underline{\omega}_C^{\otimes k} \right\rbrace,$$
where $\underline{\omega}_C$ denotes the usual canonical line bundle on a stable curve.
An extension of this cycle to the whole space $\overline{\cM}_{g,n}$ has been worked out in \cite{PandhFark} by constructing a closed substack $\widetilde{\cH}^k_g(\vec{\mathbf{m}})$ of $\overline{\cM}_{g,n}$ parametrizing the new notion of a $k$-twisted canonical divisor, see \cite[Definition 20]{PandhFark} or Section \ref{twiscan} for the definition.
In particular, it contains the closure $\overline{\cH}^k_g(\vec{\mathbf{m}})$ of $\cH^k_g(\vec{\mathbf{m}})$ into $\overline{\cM}_{g,n}$.
By \cite[Theorem 21]{PandhFark}, all the irreducible components of $\widetilde{\cH}^k_g(\vec{\mathbf{m}})$ are of dimensions at least $2g-3+n$.
In the special case $k=1$, we have two distinct situations:
\begin{itemize}
\item in the so-called strictly meromorphic case, i.e.~at least one integer $\mathbf{m}_i$ is negative, then the stack $\widetilde{\cH}^1_g(\vec{\mathbf{m}})$ is of pure dimension $2g-3+n$, and contains $\overline{\cH}^1_g(\vec{\mathbf{m}})$ as a union of irreducible components, see \cite[Theorem 3]{PandhFark},
\item in the so-called holomorphic case, i.e.~every integer $\mathbf{m}_i$ is non-negative, the stack $\overline{\cH}^1_g(\vec{\mathbf{m}}) \subset \widetilde{\cH}^1_g(\vec{\mathbf{m}})$ is of pure dimension $2g-2+n$, and all the other irreducible components of $\widetilde{\cH}^k_g(\vec{\mathbf{m}})$ are of pure dimension $2g-3+n$, see \cite[Theorem 2]{PandhFark}.
\end{itemize}
Furthermore, a conjectural relationship between a weighted fundamental class $H_g(\vec{\mathbf{m}})$ of the space $\widetilde{\cH}^1_g(\vec{\mathbf{m}})$ and the generalized Pixton's formula has been proposed in \cite[Appendix A.3]{PandhFark} in the so-called strictly meromorphic case with $k=1$.

The present paper aims to provide a geometric meaning of the weighted fundamental class $H_g(\vec{\mathbf{m}})$ using Gromov--Witten theory on a natural moduli space, in the spirit of the definition of the double ramification cycle via stable maps to rubber $\PP^1$.
Actually, relative Gromov--Witten theory as defined in \cite{Li01,Li02} works only for a target space with mild singularities, such as two smooth varieties meeting along a smooth divisor, and it involves expansions of the target space.
We consider here another approach using logarithmic geometry, and without expansions of the target.
Moduli spaces of log-stable maps and Gromov--Witten theory of log-varieties have been developed in \cite{QChen,QChen2,QCheneval} and in \cite{GrossSieb}.
Log-Gromov--Witten theory is still defined for smooth objects, but the advantage of working in the category of log-schemes is that there are more smooth objects. For instance, given a variety with a normal crossing divisor, there is a natural way to endow it with a log-structure, and if the variety is smooth outside the divisor, then the log-variety we obtain is log-smooth. In particular, the logarithmic approach covers the cases considered in \cite{Li01,Li02}, and \cite[Theorem 1.1]{Wise} shows that these two approaches are compatible.

In this paper, we define the notion of a $k$-log-canonical divisor.
Guided by the case $k=0$, where we view a map from a curve $C$ to the target space $\PP^1$ as a section of the trivial $\PP^1$-bundle over $C$, we define a $k$-log-canonical divisor as a log-curve $(C,\cM_C)$ together with a log-section $(f,f^\flat) \colon (C,\cM_C) \to (P,\cM_P)$ of the projective bundle
$$P := \mathrm{Proj}(\cO_C \oplus \omega_C^{\otimes k}),$$
endowed with the log-structure obtained by the divisor $\left[0\right]+\left[\infty\right]$ and by the pull-back of the log-structure $\cM_C$.
Here, the line bundle $\omega_C$ is the log-canonical line bundle of the log-curve $(C,\cM_C)$, and equals
$$\omega_C := \omega_{(C,\cM_C)} \simeq \underline{\omega}_C(\sigma_1+\dotsb+\sigma_n).$$
We relate the notion of $k$-log-canonical divisors with the log-maps described in \cite{QChen}, and we impose two further conditions: stability and minimality.
The stability condition is on the schematic map $f$ and forbids the appearance of non-trivial infinitesimal automorphisms.
The minimality condition is on the log-structure and was introduced in \cite[Definition 3.5.1]{QChen}.

\begin{thm}[see Theorem \ref{DMstack}]
The moduli space $\cD_{g,n}^{(k)}$ of minimal and stable $k$-log-canonical divisors is a proper Deligne--Mumford stack carrying a perfect obstruction theory.
\end{thm}

The moduli space $\cD_{g,n}^{(k)}$ is the analogue of relative stable maps, but we still need to impose rubberness, i.e.~to consider $k$-log-canonical divisors up to a $\CC^*$-action on the projective bundle.
It is done indirectly by considering a closed substack $\cD_{g,n}^{*(k)} \subset \cD_{g,n}^{(k)}$ determined by the $\CC^*$-admissibility condition, see Definition \ref{Cstaradm}.
A $\CC^*$-admissible, minimal, and stable $k$-log-canonical divisor is also called a rubber $k$-log-canonical divisor.
Then, forgetting the log-structures, the section, and stabilizing the source curve gives a natural map
$$\mathrm{st} \colon \cD_{g,n}^{*(k)} \to \overline{\cM}_{g,n}.$$
Moreover, the stack $\cD_{g,n}^{*(k)}$ has distinct open and closed components obtained by specifying integers $\vec{\mathbf{c}}=(\mathbf{c}_1,\dotsc,\mathbf{c}_n) \in \ZZ^n$ at the markings, see Section \ref{topoltype}.

\begin{thm}[see Theorem \ref{DMstar}]
Fix a partition $\vec{\mathbf{c}}=(\mathbf{c}_1,\dotsc,\mathbf{c}_n)$ of $k(2g-2+n)$, and write $\mathbf{m}_i := \mathbf{c}_i-1$.
The moduli space $\cD_{g,n,\vec{\mathbf{c}}}^{*(k)}$ is a proper Deligne--Mumford stack carrying a perfect obstruction theory.
Furthermore, all irreducible components of $\cD_{g,n,\vec{\mathbf{c}}}^{*(k)}$ are at least of the expected virtual dimension, and we have a surjective map
$$\mathrm{st} \colon \cD_{g,n,\vec{\mathbf{c}}}^{*(k)} \twoheadrightarrow \widetilde{\cH}^k_g(\vec{\mathbf{m}}) \subset \overline{\cM}_{g,n}.$$
\end{thm}

The moduli space $\cD^{*(k)}_{g,n,\vec{\mathbf{c}}}$ is not reduced in general, and we compute the multiplicities of its irreducible components.
Moreover, it has a principal component $\cD_{\mathrm{pr},(g,n,\vec{\mathbf{c}})}^{*(k)}$, which is also mapping surjectively to $\widetilde{\cH}^k_g(\vec{\mathbf{m}})$, and leading to the following theorem.

\begin{thm}[see Corollary \ref{weighted}]
In the strictly meromorphic case with $k=1$, the virtual fundamental cycle of the moduli space $\cD_{\mathrm{pr},(g,n,\vec{\mathbf{c}})}^{*(1)}$ of principal rubber $1$-log-canonical divisors equals the fundamental cycle and we obtain
$$\mathrm{st}_* \left[ \cD_{\mathrm{pr},(g,n,\vec{\mathbf{c}})}^{*(1)} \right] = H_{g}(\vec{\mathbf{m}}) \in A^g(\overline{\cM}_{g,n}),$$
where the class $H_{g}(\vec{\mathbf{m}})$ is the weighted fundamental class from \cite[Appendix A.4]{PandhFark}.
In particular, \cite[Conjecture A]{PandhFark} is rephrased as
$\mathrm{st}_* \left[ \cD_{\mathrm{pr},(g,n,\vec{\mathbf{c}})}^{*(1)} \right] = 2^{-g} P^g_{g,\vec{\mathbf{m}}},$
where the class $P^g_{g,\vec{\mathbf{m}}}$ is Pixton's cycle for $k=1$.
\end{thm}

\noindent
\textbf{Acknowledgement.}
The author is extremely grateful to Rahul Pandharipande for motivating discussions and invaluable advice on this paper. He also thanks Daniele Agostini, Ignacio Barros, Dawei Chen, Gavril Farkas, Felix Janda, Hsueh-Yung Lin, Adrien Sauvaget, and Jonathan Wise for fruitful discussions. He thanks Gavril Farkas and Rahul Pandharipande for the organization of the Einstein foundation's conference ``Moduli spaces of holomorphic differentials'', and Dan Abramovich for his lecture on logarithmic geometry during the IRTG summer school ``Moduli and Automorphic Forms''. It is also a pleasure to acknowledge Qile Chen for his illuminating paper \cite{QChen}.
Eventually, the paper is part of a bigger project aiming at proving \cite[Conjecture A]{PandhFark} and including several collaborators.
The author was supported by the Einstein foundation.

\section{The moduli space}
In this section, we define the moduli space of log-maps relevant to the study of twisted differentials defined by \cite{PandhFark}.

\subsubsection{Recall on log-geometry}
In the whole paper, we work in the algebraic category and over $\CC$.
Let $X$ be a scheme.
A pre-log-structure on $X$ is an \'etale sheaf of monoids $\cM_X$ together with a morphism of sheaves of monoids
$$\exp_X \colon \cM_X \to \cO_X,$$
where the structure sheaf $\cO_X$ is a monoid under multiplication.
We usually write $\cM_X$ instead of $(\cM_X,\exp_X)$ when no confusions arise.
A log-structure on $X$ is a pre-log-structure $(\cM_X,\exp_X)$ such that
$$\exp_X \colon \exp_X^{-1}(\cO^*_X) \simeq \cO_X^*$$
is an isomorphism.
Thus, we identify $\cO^*_X$ as a subsheaf of $\cM_X$ and we call characteristic sheaf, or ghost sheaf, the quotient
$$\overline{\cM}_X := \cM_X / \cO^*_X.$$
The trivial log-structure on a scheme $X$ is given by the inclusion $\cO_X^* \subset \cO_X$.
To any pre-log-structure, there is a canonically associated log-structure.
Then, given a morphism of scheme $f \colon X \to Y$, we can define the pull-back of a log-structure $\cM_Y$ on $Y$ by taking the log-structure $f^*(\cM_X)$ associated to the pre-log-structure $f^{-1}(\cM_Y) \to f^{-1}(\cO_Y) \to \cO_X$.
An important example of a log-structure is the divisorial log-structure. Take $X$ a regular scheme and $D \subset X$ a divisor. Then we consider the log-structure associated to
$$\cM_X(U) = \left\lbrace g \in \cO_X(U) | g_{|U-D} \in \cO^*_X(U-D) \right\rbrace \hookrightarrow \cO_X(U).$$

A morphism of log-scheme
$$(X,\cM_X) \to (Y,\cM_Y)$$
is the data of a morphism of schemes $f \colon X \to Y$ and of a morphism of sheaves of monoids $f^\flat \colon f^*(\cM_Y) \to \cM_X$ making the following diagram commutative:
\begin{center}
\begin{tikzpicture}[scale=1]
\draw (0,0) node(A1){$f^*\cM_Y$}
 (4,0) node(A2){$\cM_X$}
 (2,-1) node(A3){$\cO_X$}
 (2,-0.5) node{$\circlearrowright$}
 (1,-0.7) node[left]{$\exp_Y$}
 (3,-0.7) node[right]{$\exp_X$}
 (2,0) node[above]{$f^\flat$};
 
\draw[->,>=stealth] (A1) to[bend right=0] (A2);
\draw[->,>=stealth] (A1) to[bend right=0] (A3);
\draw[->,>=stealth] (A2) to[bend right=0] (A3);
\end{tikzpicture}
\end{center}
The morphism $f$ is called the schematic morphism and we say that the log-morphism $(f,f^\flat)$ is strict if the morphism of log-structures $f^\flat$ is an isomorphism.

A chart of a log-structure $\cM_X$ on a scheme $X$ is a constant sheaf of monoid $P$ on $X$ together with a morphism $P \to \cM_X$ such that the log-structure $\cM_X$ is associated to the pre-log-structure $P \to \cM_X \to \cO_X$.
The monoid $P$ is called
\begin{itemize}
\item integral if the natural map $P \to P^\mathrm{gp}$ of the monoid to its associated group is an injection,
\item saturated if it is integral and if every $p \in P^\mathrm{gp}$ such that $n \cdot p \in P$ for some $n \in \NN^*$ is in $P$.
\item coherent if it is finitely generated
\item fine if it is integral and coherent
\item sharp if the identity element $0$ is the only invertible element in $P$.
\end{itemize}
A log-structure $\cM_X$ on a scheme $X$ is called fine or saturated if \'etale locally there is a chart $P \to \cM_X$, with $P$ a fine or saturated monoid.

\subsubsection{Notations}
Usually, we will denote by curly letters as $\cC$ the Deligne--Mumford stacks and by gothic letters as $\kC$ the Artin stacks.
The log-structure of a scheme (or a stack) $X$ is denoted by $\cM_X$; no confusion with the stack $\overline{\cM}_{g,n}$ of stable curves should occur.
Depending on the context, the notation $X$ holds for the scheme or for the log-scheme. In the first case, we would write $(X,\cM_X)$ for the log-scheme, and in the second case, we write $\underline{X}$ for the underlying scheme of a log-scheme.
The non-negative integer $k$ is fixed in the whole paper, except when specified otherwise.

\subsection{Log-curves and $k$-log-canonical divisors}
On any family of prestable curves, there is a standard log-structure induced by the sections of the marked points and by the boundary divisor of the moduli space of prestable curves $\mathfrak{M}_{g,n}$.

Precisely, the moduli space $\mathfrak{M}_{g,n}$ is endowed with the divisorial log-structure $\cM_{\mathfrak{M}_{g,n}}$ given by the boundary divisor $\mathfrak{M}_{g,n}^\mathrm{sing}$ of singular curves. Note that it is a normal crossing divisor.
Then the log-structure on the universal curve $\mathfrak{C}_{g,n}$ is
\begin{equation*}
\cM_{\mathfrak{C}_{g,n}} := \cM_{\mathfrak{C}_{g,n}}^\# \oplus_{\cO^*_{\mathfrak{C}_{g,n}}} \cM_{\mathfrak{C}_{g,n}}^\Sigma,
\end{equation*}
where $\cM_{\mathfrak{C}_{g,n}}^\#$ is the log-structure given by the pre-image of the boundary $\mathfrak{M}_{g,n}^\mathrm{sing}$, which is also a normal crossing divisor, and $\cM_{\mathfrak{C}_{g,n}}^\Sigma$ is the log-structure given by the divisor $\Sigma=\sigma_1+\dotsb+\sigma_n$ of the marked points.

Consider a flat family $\pi \colon C \rightarrow S$ of prestable curves of genus $g$ with $n$ markings over a base scheme $S$.
It corresponds to a morphism $\pi \colon S \rightarrow \mathfrak{M}_{g,n}$ and we define the standard log-structure $\cM_S^{C/S}$ on the base $S$ as the pull-back
$$\cM_S^{C/S} := \pi^* \cM_{\mathfrak{M}_{g,n}}.$$
Similarly, we define the standard log-structure $\cM_C^{C/S}$ on the curve $C$ by
$$\cM_C^{C/S} := \pi^* \cM_{\mathfrak{C}_{g,n}},$$
and we get a log-map
\begin{equation*}
(C,\cM_C^{C/S}) \rightarrow (S,\cM_S^{C/S}).
\end{equation*}

\begin{rem}
When $S = \mathrm{Spec}(k)$ is a geometric point, the characteristic sheaf $\overline{\cM}^{C/S}_S$ is isomorphic to $\mathbb{N}^m$, where $m$ is the number of nodes in the curve $C$.
\end{rem}

Consider another log-structure $\cM_S$ on the base scheme $S$, together with a morphism of log-structures $\cM_S^{C/S} \rightarrow \cM_S$.
By taking the identity morphism $S \rightarrow S$ as the schematic map, we therefore obtain a log-map $(S,\cM_S) \rightarrow (S,\cM_S^{C/S})$, and we form the cartesian product
\begin{center}
\begin{tikzpicture}[scale=1]
\draw (0,2) node(A1){$(C,\cM_C)$}
 (2+2,2) node(A2){$(C,\cM_C^{C/S})$}
 (2+2,0) node(A3){$(S,\cM_S^{C/S})$}
 (0,0) node(A4){$(S,\cM_S)$}
 (1-0.1+1,1-0.1)--(1+0.1+1,1-0.1)--(1+0.1+1,1+0.1)--(1-0.1+1,1+0.1)--(1-0.1+1,1-0.1);
 
\draw[->,>=stealth] (A1) to[bend right=0] (A2);
\draw[->,>=stealth] (A1) to[bend right=0] (A4);
\draw[->,>=stealth] (A2) to[bend right=0] (A3);
\draw[->,>=stealth] (A4) to[bend right=0] (A3);
\end{tikzpicture}
\end{center}

\begin{dfn}
A log-curve $(C,\cM_C)$ over a fine log-scheme $(S,\cM_S)$ is a log-map $(C,\cM_C) \rightarrow (S,\cM_S)$ obtained as above.
In particular, it is determined by a usual family of marked curves together with the morphism of log-structures $\cM_S^{C/S} \rightarrow \cM_S$.
We sometimes denote a log-curve by $$(C \rightarrow S, (\sigma_1,\dotsc,\sigma_n), \cM_S^{C/S} \rightarrow \cM_S).$$
We call this log-curve prestable if the log-structure $\cM_S$ is (fine and) saturated.
\end{dfn}

Let $\pi \colon (C,\cM_C) \rightarrow (S,\cM_S)$ be a prestable log-curve.
By \cite[Proposition 4.9]{ACbook}, the log-canonical line bundle $\omega_\pi$ is naturally isomorphic to
$$\omega_\pi \simeq \underline{\omega}_{\pi}(\sigma_1+\dotsb+\sigma_n),$$
where the sheaf $\underline{\omega}_\pi$ is the usual relative canonical line bundle of a family of schematic curves.
For any fixed non-negative integer $k$, we define the scheme
\begin{equation*}
P^{(k)}_\pi := \mathrm{Proj}(\cO_C \oplus \omega_\pi^k).
\end{equation*}
The integer $k$ is fixed once for all, and when no confusions arise, we simply write $P$ instead of $P^{(k)}_\pi$.
This scheme comes together with a projection morphism $$p \colon P \rightarrow C,$$ with two sections
$$0 \colon C \to P \quad \textrm{and} \quad \infty \colon C \to P,$$
and with a log-structure
\begin{equation*}
\cM_P := p^*(\cM_C) \oplus_{\cO^*_P} \cM^{0,\infty},
\end{equation*}
where the log-structure $\cM^{0,\infty}$ is associated to the divisor $\left[0\right] + \left[\infty\right]$.
Using the natural map $p^*(\cM_C) \rightarrow \cM_P$, the projection map $p \colon (P,\cM_P) \rightarrow (C,\cM_C)$ is then a log-morphism.

\begin{dfn}
A $k$-log-canonical divisor on a prestable log-curve $(C,\cM_C) \rightarrow (S,\cM_S)$ is a log-section $f$ of the above defined $\PP^1$-bundle $(P,\cM_P)$ over $(C,\cM_C)$, i.e.~a log-map $f \colon (C,\cM_C) \rightarrow (P,\cM_P)$ such that the composition $p \circ f$ is the identity.
\end{dfn}
A $k$-log-canonical divisor is thus completely determined by a schematic section $C \xrightarrow{f} P$ together with a morphism of log-structures
\begin{center}
\begin{tikzpicture}[scale=1]
\draw (0,0) node(A1){$f^*\cM^{0,\infty}$}
 (4,0) node(A2){$\cM_C$}
 (2,-1) node(A3){$\cO_C$}
 (2,-0.5) node{$\circlearrowright$}
 (2,0) node[above]{$f^\flat$};
 
\draw[->,>=stealth] (A1) to[bend right=0] (A2);
\draw[->,>=stealth] (A1) to[bend right=0] (A3);
\draw[->,>=stealth] (A2) to[bend right=0] (A3);
\end{tikzpicture}
\end{center}
In the whole text, the $k$-log-canonical divisor $\xi$ corresponds to the tuple
\begin{equation*}
\xi = (C \rightarrow S, (\sigma_1,\dotsc,\sigma_n), \cM_S^{C/S} \rightarrow \cM_S ~ ; ~~  C \xrightarrow{f} P, f^*\cM^{0,\infty} \xrightarrow{f^\flat} \cM_C).
\end{equation*}
Furthermore, it will be convenient to denote by
$$\underline{\xi}=(C \rightarrow S, (\sigma_1,\dotsc,\sigma_n); ~~  C \xrightarrow{f} P)$$
the underlying family of schematic maps.
Recall that a schematic section $f \colon C \rightarrow P$ is equivalent to a couple $(\cL,s)$, where $\cL$ is a line bundle on $C$ and $s$ is a non-vanishing global section $$s \in H^0(C,\cL \oplus \omega_\pi^{-k} \otimes \cL).$$

\begin{rem}\label{DF}
It is important to keep in mind the following equivalent definition of a $k$-log-canonical divisor over a base scheme $S$.
First, we endow the projective bundle $P$ with the log-structure $\cM^{0,\infty}$ and then we require that the log-map $f \colon C \to P$ is a section of the schematic projection $p \colon P \to C$ only at the level of schemes. Indeed, the existence of a projection log-map $(P,\cM^{0,\infty}) \to (C,\cM_C)$ is not granted anymore, but there are still log-morphisms $f \colon (C,\cM_C) \to (P,\cM^{0,\infty})$ such that the schematic morphism is a section of the schematic projection $p \colon P \to C$.
This description of a $k$-log-canonical divisor has a huge advantage: the target $(P,\cM^{0,\infty})$ is a Deligne--Faltings pair of rank one, see \cite[Definition 6.1]{ACbook}.
Indeed, the log-structure $\cM^{0,\infty}$ is associated to a global section of the line bundle $\cO_P(\left[0\right]+\left[\infty\right])$ over $P$ vanishing at $0$ and at $\infty$.
Moreover, the scheme $P$ is a projective variety over the base scheme $S$. 
Therefore, we can use the results from \cite{QChen} to study $k$-log-canonical divisors.
\end{rem}

\subsection{Automorphisms}\label{auto}
An isomorphism between two log-curves $(C_1,\cM_{C_1}) \rightarrow (S_1,\cM_{S_1})$ and $(C_2,\cM_{C_2}) \rightarrow (S_2,\cM_{S_2})$ over the same base scheme $S=S_1=S_2$ is a pair $(\rho,\theta)$ such that
\begin{center}
\begin{tikzpicture}[scale=0.75]
\draw (0,0) node(A1){$(C_1,\cM_{C_1})$}
 (0,-2) node(A3){$(S_1,\cM_{S_1})$};
\draw[->,>=stealth] (A1) to[bend right=0] (A3);
\draw (4,0) node(B1){$(C_2,\cM_{C_2})$}
 (4,-2) node(B3){$(S_2,\cM_{S_2})$};
\draw[->,>=stealth] (B1) to[bend right=0] (B3);
\draw[->,>=stealth] (A1) to[bend right=0] (B1);
\draw[->,>=stealth] (A3) to[bend right=0] (B3);
\draw (2+0.2,+0.2) node[left]{$\rho$}
      (2,-1) node{$\circlearrowleft$}
      (2+0.2,-2+0.2) node[left]{$\theta$};
\end{tikzpicture}
\end{center}
where $\theta \colon (S_1,\cM_{S_1}) \rightarrow (S_2,\cM_{S_2})$ and $\rho \colon (C_1,\cM_{C_1}) \rightarrow (C_2,\cM_{C_2})$ are isomorphisms of log-schemes and the induced map $\underline{\theta} \colon S_1 \rightarrow S_2$ is the identity map $S_1=S_2$.
Therefore, it is an isomorphism of marked curves together with an isomorphism of log-structures on the base scheme $S$:
\begin{center}
\begin{tikzpicture}[scale=0.75]
\draw (-1,0) node(A1){$(C_1,\left\lbrace \sigma_i \right\rbrace)$}
 (3,0) node(A2){$(C_2,\left\lbrace \sigma'_i \right\rbrace)$}
 (1,-2) node(A3){$S$};
\draw (1,0) node[above]{$\simeq$};
\draw (1,0.3) node[above]{$\rho$};
\draw[->,>=stealth] (A1) to[bend right=0] (A2);
\draw[->,>=stealth] (A1) to[bend right=0] (A3);
\draw[->,>=stealth] (A2) to[bend right=0] (A3);
\draw     (1,-1) node{$\circlearrowright$};
\draw (1+5+0,-2) node(B1){$\cM_{S_2}$}
 (1+5+3,-2) node(B2){$\cM_{S_1}$}
 (1+5+3,0) node(B3){$\cM^{C_1/S}_S$}
 (1+5,0) node(B4){$\cM^{C_2/S}_S$};
\draw (1+5+1.5,-2) node[below]{$\simeq$};
\draw (1+5+1.5,-2.3) node[below]{$\theta$};
\draw[->,>=stealth] (B1) to[bend right=0] (B2);
\draw[->,>=stealth] (B4) to[bend right=0] (B3);
\draw[->,>=stealth] (B3) to[bend right=0] (B2);
\draw[->,>=stealth] (B4) to[bend right=0] (B1);
\draw     (1+5+1.5,-1) node{$\circlearrowleft$};
\end{tikzpicture}
\end{center}

Observe that the isomorphism $\rho$ leads to natural isomorphisms
\begin{equation*}
\omega_{C_1/S} \simeq \rho^*\omega_{C_2/S} \quad \textrm{and} \quad \cO_{C_1} \simeq \rho^*\cO_{C_2},
\end{equation*}
so that it induces an isomorphism $\rho^{\textrm{ind}}$
\begin{center}
\begin{tikzpicture}[scale=0.75]
\draw (0,0) node(A1){$P_1$}
 (3,0) node(A2){$P_2$}
 (0,-2) node(A3){$C_1$}
 (3,-2) node(A4){$C_2$};
\draw (1.5,0) node[above]{$\simeq$};
\draw (1.5,0.3) node[above]{$\rho^{\textrm{ind}}$}; 
\draw (1.5,-2) node[below]{$\simeq$};
\draw (1.5,-2.3) node[below]{$\rho$};
\draw[->,>=stealth] (A1) to[bend right=0] (A2);
\draw[->,>=stealth] (A1) to[bend right=0] (A3);
\draw[->,>=stealth] (A2) to[bend right=0] (A4);
\draw[->,>=stealth] (A3) to[bend right=0] (A4);
\draw     (1.5,-1) node{$\square$};
\end{tikzpicture}
\end{center}
respecting (by pull-back) the sections $0$ and $\infty$ and thus the log-structures
$$\cM^{0,\infty}_{P_1} = \rho^{\textrm{ind}~*} (\cM^{0,\infty}_{P_2}).$$

\begin{dfn}\label{auto1}
An isomorphism between two $k$-log-canonical divisors $\xi_1$ and $\xi_2$ over the same base scheme $S$ is an isomorphism of log-curves $(\rho, \theta)$  such that
$$f_1 = \rho^{\textrm{ind}~-1} \circ f_2 \circ \rho.$$
We denote by $\mathcal{A}ut_S(\xi)$ the functor of automorphisms of $\xi$ over $S$, and by $\mathcal{A}ut_S(\underline{\xi})$ the induced functor of automorphisms of the schematic map $\underline{\xi}$ over $S$.
\end{dfn}

\begin{rem}\label{DFaut}
Following Remark \ref{DF}, a $k$-log-canonical divisor $\xi$ over a base scheme $S$ is a particular example of a log-map from a family of log-curves to the log-target $(P,\cM^{0,\infty})$, which is a projective Deligne--Faltings pair over the scheme $S$, hence an object of the category $\mathcal{L}\cM_{g,n}((P,\cM^{0,\infty})/S)(S)$ considered in \cite[Definition 2.1.5]{QChen}.
However, the automorphism group of a $k$-log-canonical divisor $\xi$ over a base scheme $S$ depends on whether we consider $\xi$ as a log-map or as a log-section.
The reason is that an automorphism as a log-map is a pair $(\rho,\theta)$ such that
$$f = f \circ \rho,$$
using notations of Definition \ref{auto1}. Therefore, at the scheme level and taking the composition with the projection $p \colon P \to C$, we get
$\mathrm{id} = \rho$, which is false in general for the group $\mathcal{A}ut_S(\xi)$.
Moreover, note that an automorphism of $\xi$ as a log-section is just the data of an automorphism $\rho$ of marked curves satisfying the relation
$$\rho^{\textrm{ind}} \circ f = f \circ \rho$$
on schemes, so that we have the inclusion $\mathcal{A}ut_S(\xi) \subset \mathcal{A}ut_S(C)$ inside the automorphisms of usual families of marked curves.\end{rem}

\subsection{Minimality and stability conditions}\label{MinStab}
By Remark \ref{DF}, we can use the same notion of minimality as for log-maps, see \cite[Definition 3.5.1]{QChen}.
Recall that it is valid, since on any geometric point $s = \mathrm{Spec}(k)$, the log-scheme $(P_s,\cM_s^{0,\infty})$ is a Deligne--Faltings pair of rank $1$ and the scheme $P_s$ is a projective variety.
The notion for stability is similar to \cite[Definition 3.6.1]{QChen} and will be treated at the end of this section.

\begin{dfn}
A $k$-log-canonical divisor $\xi$ over a scheme $S$ is minimal if, for every geometric point $s \in S$, the induced log-section 
\begin{center}
\begin{tikzpicture}[scale=1]
\draw (0,0) node(A1){$(C_s,\cM_{C_s})$}
 (4,0) node(A2){$(P_s,\cM_{P_s})$}
 (2,-1) node(A3){$(s,\cM_s)$}
 (2,0) node[above]{$f_s$};
 
\draw[->,>=stealth] (A1) to[bend right=0] (A2);
\draw[->,>=stealth] (A1) to[bend right=0] (A3);
\draw[->,>=stealth] (A2) to[bend right=0] (A3);
\end{tikzpicture}
\end{center}
on the fiber is a minimal log-map according to \cite[Definition 3.5.1]{QChen}.
\end{dfn}

\begin{rem}
Minimality roughly means that we keep the log-structure $\cM_s$ on the geometric point $s$ as minimal as possible.
Note that we already had a condition on the log-structure $\cM_S$ for a curve to be prestable: the log-structure $\cM_S$ is fine and saturated.
\end{rem}

Let us recall the minimality condition by describing a $k$-log-canonical divisor in more details.
Fix a $k$-log-canonical divisor
\begin{equation*}
\xi = (\pi \colon C \rightarrow S, (\sigma_1,\dotsc,\sigma_n), \cM_S^{C/S} \rightarrow \cM_S ~ ; ~~  C \xrightarrow{f} P, f^*\cM^{0,\infty} \xrightarrow{f^\flat} \cM_C)
\end{equation*}
over a geometric point $S=\mathrm{Spec}(k)$, where the log-structure $\cM_S$ is fine and saturated, and take a geometric point $x \in C$.
On one hand, the stalk $f^*\overline{\cM}_x^{0,\infty}$ of the characteristic sheaf is
\begin{equation*}
f^*\overline{\cM}_x^{0,\infty} = \left\lbrace
\begin{array}{ll}
\mathbb{N} & \textrm{if } f(x) \in \left\lbrace 0,\infty \right\rbrace, \\
0 & \textrm{otherwise;}
\end{array}
\right.
\end{equation*}
denote by $\delta$ a generator of this monoid.
On the other hand, by \cite[Theorem 1.3]{Kato}, the stalk $\overline{\cM}_{C,x}$ on $x \in C$ is
\begin{equation*}
\overline{\cM}_{C,x} = \left\lbrace
\begin{array}{ll}
\pi^*\cM_S \oplus \mathbb{N} & \textrm{if } x \in \left\lbrace \sigma_1,\dotsc,\sigma_n \right\rbrace, \\
\pi^*\cM_S \oplus \mathbb{Z} & \textrm{if } x \textrm{ is a node}, \\
\pi^*\cM_S & \textrm{otherwise.}
\end{array}
\right.
\end{equation*}
At a marked point $x=\sigma_i$, we denote by $\epsilon_i$ the generator of $\mathbb{N}$.
At a node $p$, we have two generators $\epsilon_{p'}$ and $\epsilon_{p''}$ of $\mathbb{Z}$ corresponding to the choice of a branch at the node.
There is also a special element $e_p \in \pi^* \cM_S$ smoothing the node $p$ and satisfying
$$e_p = \epsilon_{p'} + \epsilon_{p''}.$$
Therefore, at the level of characteristics and on the stalk of $x \in C$, the morphism of log-structures $f^\flat$ is
$$\overline{f^\flat}_x(\delta)=e_x + \left\lbrace
\begin{array}{lll}
c_i \epsilon_i & \textrm{if } x=\sigma_i, & \textrm{ with } c_i \in \mathbb{N}, \\
c_{p'} \epsilon_{p'} & \textrm{if } x \textrm{ is a node } p, & \textrm{ with } c_{p'} \in \mathbb{Z}, \\
0 & \textrm{otherwise,}& 
\end{array}
\right.$$
with $e_x \in \pi^* \cM_S$.

\begin{rem}
By \cite[Lemma 3.5]{Olsson2}, the element $e_x \in \pi^* \cM_S$ is  independent of the choice of a smooth point $x \in C_v$, for a given irreducible component $C_v$ of the curve $C$; it is called the degeneracy of the component $C_v$ and it is denoted by $e_v$.
When $x$ is a node $p$, then the element $e_x$ in the previous equation corresponds to the degeneracy of the irreducible component corresponding of the chosen branch $p'$ at the node.
From the equation $e_p = \epsilon_{p'} + \epsilon_{p''}$, we see that 
$$e_x + c_{p'} \epsilon_{p'} = e_x + c_{p'} e_p - c_{p'} \epsilon_{p''} = e_x + c_{p''} \epsilon_{p''},$$
so that we get
$$c_{p'} + c_{p''} = 0 \in \mathbb{Z}$$
and the degeneracy on the component corresponding to the other branch $p''$ is
$$e_x + c_{p'} e_p.$$
\end{rem}

\begin{rem}\label{specialpoints}
Obviously, when the element $\delta$ is zero, then its image is zero.
Conversely, from the definition of a log-structure, we know that $\overline{f^\flat}_x(\delta)$ cannot be zero if $\delta$ is non-zero.
Therefore, the degeneracy $e_v$ of an irreducible component $C_v$ is non-zero if and only if the schematic map $f_v \colon C_v \to P$ is the constant map equal to one of the sections $0$ or $\infty$.
Furthermore, if the degeneracy $e_v$ is zero, then the schematic map $f$ cannot be equal to $0$ or to $\infty$ on $C_v$ outside of the marked points and of the nodes.
In that case, the integer $c_i$ is exactly the contact order of $f(C_v)$ with the $0$ or $\infty$ section at the marked point $\sigma_i$.
\end{rem}

\begin{rem}
At a node $p$ where two components $C_{v'}$ and $C_{v''}$ meet, if the integer $c_{p'}$ is non-zero, then the schematic section $f$ has to meet the divisor $0$ or the divisor $\infty$ at $p$.
Assuming $e_{v'} = 0$, i.e.~the schematic map $f_{|C_{v'}}$ is not the $0$ or the $\infty$-section, then the integer $c_{v'}$ is the contact order of $f_{|C_{v'}}$ with either $0$ or $\infty$ at $p$, so it needs to be positive.
Therefore, the integer $c_{p''}$ is negative and $e_{v''} \neq 0$, i.e.~the schematic map $f_{|C_{v''}}$ is the $0$ or the $\infty$ section.
\end{rem}

Now, we attach the following data, called a marked graph, to the $k$-log-canonical divisor $\xi$:
\begin{itemize}
\item the dual graph $G_\xi$ of the source curve,
\item an element $\mathbf{s}_v \in \{+,-,0\}$, called a sign, for each vertex $v$
\item an element $\mathbf{s}_h \in \{+,-,0\}$ for each half-edge $h$,
\item the integer $c_h$, called the contact order, for each half-edge $h$,
\end{itemize}
where the sign $\mathbf{s}_v$ of a vertex $v$ corresponding to an irreducible component $C_v$ is
$$\mathbf{s}_v = \left\lbrace \begin{array}{ll}
+ & \textrm{if the schematic map $f$ is identically $0$ on $C_v$,} \\
- & \textrm{if the schematic map $f$ is identically $\infty$ on $C_v$,} \\
0 & \textrm{otherwise, i.e.~if $e_v=0$,} \\
\end{array}
\right.$$
and the sign $\mathbf{s}_h$ of a half-edge $h$ corresponding to a marking or a node $x$ is
$$\mathbf{s}_h = \left\lbrace \begin{array}{ll}
+ & \textrm{if $f(x)=0$,} \\
- & \textrm{if $f(x)=\infty$,} \\
0 & \textrm{otherwise, i.e.~if $c_h=0$.} \\
\end{array}
\right.$$

We observe the following compatibility rules:
\begin{itemize}
\item[-]for any leg $h$ (half-edge corresponding to a marked point), we have $c_h \in \NN$,
\item[-]if two half-edges $h,h'$ form an edge, then
$$\mathbf{s}_h=\mathbf{s}_{h'} \quad \textrm{and} \quad c_h+c_{h'}=0,$$
\item[-]for every half-edge $h$ incident to a vertex $v$, we have
\begin{eqnarray*}
\mathbf{s}_v \neq 0 & \implies & \mathbf{s}_h=\mathbf{s}_v, \\
c_h <0 & \implies & e_v \neq 0,
\end{eqnarray*}
\item[-]for any vertex $v$ and half-edge $h$, we have
$$\mathbf{s}_v=0 \iff e_v=0 \quad \textrm{and} \quad \mathbf{s}_h=0 \iff c_h=0,$$
but there is no relation between a non-zero $\mathbf{s}_h$ and the sign of a non-zero contact order $c_h$
\end{itemize}

Associated to the marked graph $G_\xi$ of $\xi$, we consider the monoid
$$M(G_\xi) := \langle e'_v,e'_l ~~ | ~~ v \in V(G_\xi), l \in E(G_\xi) \rangle / \langle R_v,R_l ~~ | ~~ v \in V(G_\xi), l \in E(G_\xi) \rangle,$$
where $V(G_\xi)$ is the set of vertices of the graph $G_\xi$, $E(G_\xi)$ is the set of edges, and the relations $R_v$ and $R_l$ are
\begin{equation*}
\begin{array}{ll}
R_v: e'_v = 0 & \textrm{if $\mathbf{s}_v=0$, i.e.~$e_v=0$,} \\
R_l: e'_{v_2} = e'_{v_1} + c_h e'_l & \textrm{if $l$ is the edge linking the vertices $v_1$ and $v_2$,} \\
& \textrm{$h$ is the half-edge attached to the vertex $v_1$,}\\
& \textrm{and $c_h \geq 0$.}
\end{array}
\end{equation*}
Then, we consider the associated group $M(G_\xi)^\textrm{gp}$, its torsion part $T(G_\xi)$, and $\overline{\cM}(G_\xi)$ the saturation of the image of $M(G_\xi)$ inside $M(G_\xi)^\textrm{gp}/T(G_\xi)$.
The monoid $\overline{\cM}(G_\xi)$ is called the monoid associated to the marked graph $G_\xi$.

By \cite[Proposition 3.4.2]{QChen}, since the monoid $\cM_S$ of the $k$-log-canonical divisor $\xi$ is fine and saturated, then we have a morphism of monoids
$$\overline{\cM}(G_\xi) \to \overline{\cM}_S,$$
sending the generator $e'_v$ to the degeneracy $e_v$ and the generator $e'_l$ to the element $e_l \in \pi^*\overline{\cM}_S$ smoothing the node $l$.

\begin{dfn}
We say that the $k$-log-canonical divisor $\xi$ is minimal if the monoid $\cM_S$ is fine and saturated and if the above map
$$\overline{\cM}(G_\xi) \simeq \overline{\cM}_S$$
is an isomorphism of monoids.
\end{dfn}

\begin{rem}
By \cite[Propositions 3.5.2]{QChen}, minimality is an open condition.
\end{rem}

Now, we define the stability condition for $k$-log-canonical divisors, so that they have no non-trivial infinitesimal automorphisms.

\begin{dfn}\label{stabledef}
A $k$-log-canonical divisor $\xi$ over a scheme $S$ is stable if, for every geometric point $s \in S$, the induced schematic morphism
$$f_s \colon C_s \to P_s$$
is stable as a section, i.e.~if
the group of automorphisms $\mathcal{A}ut_s(\underline{\xi_s})(s)$ is finite\footnote{We will see in Lemma \ref{repres} that the group $\mathcal{A}ut_s(\xi_s)(s)$ is also finite in that case.}.
\end{dfn}

Stability is a condition on the schematic map, viewed as a section. We have seen in Remark \ref{DFaut} that it is different from the usual stability condition of maps, for which the schematic morphism $f_s \colon C_s \to P_s$ has no non-trivial automorphisms and thus is always stable as a map.
However, it is not always stable as a section, as we see in the following lemma.

\begin{lem}\label{stable}
Let $\xi_s$ be a $k$-canonical divisor over a geometric point $s$.
The automorphism group $\mathcal{A}ut_s(\underline{\xi_s})(s)$ is finite if and only if
\begin{itemize}
\item the source curve $C_s$ has no unstable irreducible components of type $(\PP^1,0)$,
\item the map $f_s$ does not contract irreducible components of type $(\PP^1,0,\infty)$.
\end{itemize}
These two conditions are open:
if we have a $k$-log-canonical divisor $\xi$ over a scheme $S$ such that $\xi_s$ is stable for some geometric point $s \in S$, then there exists an \'etale neighborhood of $s$ on which the restriction of $\xi$ is stable.
\end{lem}

\begin{proof}
Since any automorphism $\rho \in \mathcal{A}ut_s(\underline{\xi_s})(s)$ is in particular an automorphism of marked curves, then the group $\mathcal{A}ut_s(\underline{\xi_s})(s)$ is finite when the source curves are stable.
Moreover, we can work independently on each irreducible component $C_v$ of the source curve $C$.
Therefore, we have only two cases to consider:
\begin{itemize}
\item the irreducible component $C_v$ is of type $(\PP^1,0)$, i.e.~one special point,
\item the irreducible component $C_v$ is of type $(\PP^1,0,\infty)$, i.e.~two special points.
\end{itemize}

In the first case, we have
$$P := \PP(\cO_{\PP^1} \oplus \cO_{\PP^1}(-k)) \simeq \PP(\cO_{\PP^1}(k) \oplus \cO_{\PP^1}).$$
Moreover, we have seen in Remark \ref{specialpoints} that all the preimages of the sections $0$ and $\infty$ are special points, so that a section $f \colon \PP^1 \to P$ is necessarily of the form
$$\left[x : y \right] \mapsto \left[x : y ; \alpha y^k : \beta \right]$$
in homogeneous coordinates, where $\left[\alpha : \beta \right] \in \PP^1$.
But we have
$$\left[x : y ; \alpha y^k : \beta \right] = \left[\lambda x : \lambda y ; \lambda^k \alpha y^k : \beta \right], \quad \forall \lambda \in \CC^*,$$
so that the section $f$ is constant in the open subset for which $y \neq 0$.
Consequently, every automorphism of $(\PP^1,0)$ preserves the section $f$, and the group $\mathcal{A}ut_s(\underline{\xi_s})(s)$ is infinite.

In the second case, we have $P \simeq \PP^1 \times \PP^1$ and a section $f \colon \PP^1 \to \PP^1 \times \PP^1$ is necessarily of the form
$$f \colon \left[x : y \right] \mapsto \left[x : y \right],\left[\alpha x^c : \beta y^c \right],$$
where $c \in \NN$ is the contact order at the special points and $\left[\alpha:\beta\right] \in \PP^1$.
An automorphism of $\Phi$ of $(\PP^1,0,\infty)$ is of the form
$\left[x : y \right] \mapsto \left[\mu x : \nu y \right]$, with $\left[\mu:\nu\right] \in \PP^1$,
and sends the section $f$ to
$$\Phi^*f \colon \left[x : y \right] \mapsto \left[x : y \right],\left[\mu^c \alpha x^c : \nu^c \beta y^c \right].$$
Therefore, we have $f=\Phi^*f$ if and only if $\mu^c=\nu^c$ or $c=0$.
Consequently, when $f$ is a contraction, then every automorphism of $(\PP^1,0,\infty)$ is in $\mathcal{A}ut_s(\underline{\xi_s})(s)$.
Otherwise, there are at most $c$ automorphisms.

Consider a family of $k$-log-canonical divisors $\xi$ over a base scheme $S$, such that $\xi_s$ is stable on a geometric point $s$.
Then, there is an \'etale neighborhood of $s$ where the source curve and the schematic map are not more degenerate than $C_s$ and $f_s$, so that no unstable components of type $(\PP^1,0)$ and no contractions can appear.
\end{proof}

\subsection{The topological type of a $k$-log-canonical divisor}\label{topoltype}
Recall from \cite[Convention 3.6.3]{QChen} that the log-map
\begin{equation*}
\xi = (\pi \colon C \rightarrow S, (\sigma_1,\dotsc,\sigma_n), \cM_S^{C/S} \rightarrow \cM_S ~ ; ~~  C \xrightarrow{f} P, f^*\cM^{0,\infty} \xrightarrow{f^\flat} \cM_C)
\end{equation*}
over a geometric point $S = \textrm{Spec}(k)$ is of topological type
$$\Gamma=(g,n,\beta,\vec{c})$$
where $g$ is the genus of the curve, $n$ is the number of marked points, $\beta$ is the homology class of the image
$$\beta = f(C) \in H_2(P,\ZZ),$$
and $\vec{c}=(c_1,\dotsc,c_n) \in \NN^n$ is the contact order at the marked points between the image $f(C)$ and the divisor $[0]+[\infty]$.

We see that the contact orders $c_1,\dotsc,c_n \in \NN$ come in two flavors, depending on whether the image $f(C)$ meets the divisor $[0]$ or the divisor $[\infty]$ at the marked points.
Thus, we introduce the notation
$$\mathbf{c}_i := \mathbf{s}_i \cdot c_i.$$
Accordingly, for each half-edge $h$ of the dual graph of $C$, we introduce the notation
$$\mathbf{c}_h := \mathbf{s}_h \cdot c_h.$$

\begin{rem}\label{homol}
The homology class $\beta$ of the image $f(C)$ depends only on the schematic map.
It satisfies
\begin{equation*}\label{condsec}
\beta \cdot F = 1,
\end{equation*}
where $F$ is a fiber of the projection $p \colon P \to C$.
In fact, this equation on $\beta$ characterizes the fact that the morphism $f$ is a section of the projective bundle, up to an isomorphism of the curve $C$.
Furthermore, on a non-degenerate irreducible component $C_v$, i.e.~with $e_v = 0$, the restriction $\beta_v$ of $\beta$ to $C_v$ satisfies
$$\beta \cdot [0] = \sum_{h \in I^+_v} c_h p_h \quad \textrm{and} \quad \beta \cdot [\infty] = \sum_{h \in I^-_v} c_h p_h$$
where $I^+_v$ (resp.~$I^-_v$) is the set of half-edges $h$ incident to the vertex $v$ with sign $\mathbf{s}_h =+$ (resp.~with sign $\mathbf{s}_h =-$), and $p_h$ denotes the marking or the node corresponding to $h$.
We recall that every integer $c_h$ has to be non-negative in order for the component $C_v$ to be non-degenerate.
From the three equations above on the homology class $\beta$, and from the relation
$$\cO_{C_v}([0]-[\infty]) \simeq p^*\omega_{C_v}^k,$$
we get
\begin{equation}\label{multi}
\sum_{h \in I^+_v \cup I^-_v} \mathbf{c}_h = \sum_{h \in I^+_v} c_h - \sum_{h \in I^-_v} c_h = k (2 g_v -2 + n_v),
\end{equation}
where $g_v$ is the genus of $C_v$ and $n_v$ the number of half-edges incident to $v$.
We will see a similar formula \eqref{compat} on degenerate components as well.
In that case, the only information about the schematic section $f$ is given by its sign
$$\beta = \left\lbrace \begin{array}{ll}
[0] & \textrm{if $\mathbf{s}_v=+$,} \\
\left[ \infty \right] & \textrm{if $\mathbf{s}_v=-$.}
\end{array}
\right.$$
As a consequence, we see that the homology class $\beta$ is completely determined by the marked graph $G_\xi$.
\end{rem}

\begin{rem}\label{finitemarkedgraph}
By \cite[Proposition 3.7.5]{QChen}, there are finitely many minimal and stable $k$-log-canonical divisors over a geometric point with fixed schematic map and marked graph.
Moreover, once we fix the values $c_1,\dotsc,c_n$ of the contact orders at the markings, then we have finitely many possible marked graphs.
\end{rem}

\begin{dfn}
A minimal and stable $k$-log-canonical divisor $\xi$ over a scheme $S$ is of topological type $\Gamma=(g,n,\vec{\mathbf{c}})$ if, for every geometric point $s \in S$, the induced log-section 
\begin{center}
\begin{tikzpicture}[scale=1]
\draw (0,0) node(A1){$(C_s,\cM_{C_s})$}
 (4,0) node(A2){$(P_s,\cM_{P_s})$}
 (2,-1) node(A3){$(s,\cM_s)$}
 (2,0) node[above]{$f_s$};
 
\draw[->,>=stealth] (A1) to[bend right=0] (A2);
\draw[->,>=stealth] (A1) to[bend right=0] (A3);
\draw[->,>=stealth] (A2) to[bend right=0] (A3);
\end{tikzpicture}
\end{center}
on the fiber is a genus-$g$ curve with $n$ markings $\sigma_1,\dotsc,\sigma_n$, with the contact order $|\mathbf{c}_i|$ at $\sigma_i$, and with $f(\sigma_i)=0$ if $\mathbf{c}_i>0$ and $f(\sigma_i)=\infty$ if $\mathbf{c}_i<0$.
\end{dfn}

\begin{rem}\label{logmap}
Following Remarks \ref{DF} and \ref{DFaut}, a $k$-log-canonical divisor $\xi$ over a base scheme $S$ is an object of the category $\mathcal{L}\cM_{g,n}((P,\cM^{0,\infty})/S)(S)$ considered in \cite[Definition 2.1.5]{QChen}, but the automorphisms are not the same.
Nevertheless, for any geometric point $s \in S$, the $k$-log-canonical divisor $\xi_s$ is a log-map over $s$ to the target $P_s \subset P_S$, and thus to $P_S$, the projective bundle associated to $\xi$.
Moreover, the condition for the schematic map $f_s \colon C_s \to P_s$ to be a section is an open and closed condition encoded in the homology class $\beta$ of the image, as explained in Remark \ref{homol}.
Furthermore, the minimality conditions for $k$-log-canonical divisors and for log-maps are the same, and the stability condition as a map is vacuous.
Therefore, a minimal and stable $k$-log-canonical divisor over a base scheme $S$ (and without considering automorphisms) corresponds to a $S$-point of the stack $\mathcal{K}_\Gamma(P,\cM^{0,\infty})$, introduced in \cite[Definition 3.6.5]{QChen}, for some topological type $\Gamma$.
\end{rem}

\begin{dfn}
For any topological type $\Gamma=(g,n,\vec{\mathbf{c}})$, the moduli space $\mathcal{D}_\Gamma^{(k)}$ is the fibered category over the category of schemes, for which the groupoid of $S$-points is the set of minimal and stable $k$-log-canonical divisors of topological type $\Gamma$ over $S$, together with automorphisms $\mathcal{A}ut_S(\xi)$.
\end{dfn}

We see from equations \eqref{multi} and \eqref{compat} that a necessary condition for the moduli space $\mathcal{D}_\Gamma^{(k)}$ to be non-empty is
$$\mathbf{c}_1+\dotsb+\mathbf{c}_n= k (2 g - 2 + n).$$

\begin{dfn}
For any topological type $\Gamma=(g,n,\vec{\mathbf{c}})$, the moduli space $D_\Gamma^{(k)}$ is the fibered category over the category of schemes obtained from $\mathcal{D}_\Gamma^{(k)}$ by forgetting the log-structures and keeping only the underlying schemes and morphisms, together with automorphisms $\mathcal{A}ut_S(\underline{\xi})$ of the underlying schematic map $\underline{\xi}$.
We denote the forgetful functor by
$$\mathrm{FL} \colon \mathcal{D}_\Gamma^{(k)} \to D_\Gamma^{(k)}.$$
\end{dfn}

\begin{lem}\label{repres}
For any $k$-log-canonical divisor $\xi$ over a geometric point $S$, the map of groups
$$\mathcal{A}ut_S(\xi)(S) \to \mathcal{A}ut_S(\underline{\xi})(S)$$
is an isomorphism.
\end{lem}

\begin{proof}
Take an element $\rho \in \mathcal{A}ut_S(\underline{\xi})(S)$, i.e.~an automorphism of the family of marked curves satisfying
$$\rho^{\textrm{ind}} \circ\underline{f} = \underline{f} \circ \rho,$$
where $\underline{f}$ is the schematic map from $\xi$.
We need to prove that there is exactly one element $(\rho,\theta) \in \mathcal{A}ut_S(\xi)(S)$, i.e.~an automorphism $\theta$ of  the log-structure $\cM_S$ satisfying
$$f^\flat = \widetilde{\theta} \circ f^\flat,$$
where the morphism $\widetilde{\theta} \colon \cM_C \to \cM_C$ is canonically induced by the morphism $\theta \colon \cM_S \to \cM_S$.

When $\rho$ is the identity, by Remark \ref{DFaut}, we are in the log-map case and, by \cite[Lemma 3.8.3]{QChen}, there is a unique $\theta_0$ satisfying above relation.
When the automorphism $\rho$ is not trivial, then $\rho$ induces an automorphism of $\cM_S^{C/S}$ and thus of $\cM_S$. Then, we take the composition with $\theta_0$.
\end{proof}

\section{Properties of the moduli space}
In this section, we prove the following theorem.

\begin{thm}\label{DMstack}
The moduli space $\mathcal{D}_\Gamma^{(k)}$ is a proper Deligne--Mumford stack.
Furthermore, the natural map $\mathrm{FL}$, forgetting the log-structures, is representable and finite.
Moreover, there exists a perfect obstruction theory, in the sense of \cite{BF}, on the moduli space $\mathcal{D}_\Gamma^{(k)}$, relative to the stack of prestable log-curves.
Hence it carries a virtual fundamental cycle $\left[\mathcal{D}_\Gamma^{(k)}\right]^\mathrm{vir}$.
\end{thm}

The rest of this section consists of the proof of Theorem \ref{DMstack}, based on the papers \cite{QChen,GrossSieb}.
We fix the topological type $\Gamma$ and the integer $k$, so that we denote the moduli space $\cD$.

\subsection{$\cD$ is a DM stack}
We take a morphism $s \colon S \to \mathfrak{M}_{g,n}$ from a scheme $S$, it corresponds to a family of prestable curves $s \colon C \to S$.
We denote by $P_s$ the corresponding projective bundle over $C$.
It is a projective variety over the scheme $S$.
Endowed with the divisorial log-structure $\cM^{0,\infty}$, it is a Deligne--Faltings pair of rank one.
Therefore, for each curve class $\beta \in H_2(P,\ZZ)$, we form the moduli space
$$\mathcal{K}_{s,\beta}:=\mathcal{K}_{\{\Gamma,\beta\}}((P_s,\cM^{0,\infty})/S)$$
of $\{\Gamma,\beta\}$-minimal stable log-maps from \cite[Proposition 5.7]{QChen2}.
By \cite[Proposition 5.8]{QChen2}, it is a Deligne--Mumford stack of finite type, proper over the scheme $S$. In particular, it is an algebraic stack.

We define the open substack
$$\mathcal{K}_{s,\beta}^o \subset \mathcal{K}_{s,\beta}$$
by requiring the two conditions of Lemma \ref{stable}.
We form the fibered product
\begin{center}
\begin{tikzpicture}[scale=0.75]
\draw (0,2) node(A1){$\mathcal{K}^o_{s,\beta}(s)$}
 (2+2,2) node(A2){$\mathcal{K}^o_{s,\beta}$}
 (2+2,0) node(A3){$\mathfrak{M}_{g,n}$}
 (0,0) node(A4){$S$}
 (1-0.1+1,1-0.1)--(1+0.1+1,1-0.1)--(1+0.1+1,1+0.1)--(1-0.1+1,1+0.1)--(1-0.1+1,1-0.1);
\draw (2,0) node[below]{$s$};
 
\draw[->,>=stealth] (A1) to[bend right=0] (A2);
\draw[->,>=stealth] (A1) to[bend right=0] (A4);
\draw[->,>=stealth] (A2) to[bend right=0] (A3);
\draw[->,>=stealth] (A4) to[bend right=0] (A3);
\end{tikzpicture}
\end{center}
which is also an algebraic stack over the scheme $S$.
The group scheme $\mathcal{A}ut(s)$ of automorphisms of the family $s$ of prestable marked curves acts freely on the stack $\mathcal{K}^o_{s,\beta}(s)$ by composition.
On the other hand, we form the fibered product
\begin{center}
\begin{tikzpicture}[scale=0.75]
\draw (0,2) node(A1){$\cD(s)$}
 (2+2,2) node(A2){$\cD$}
 (2+2,0) node(A3){$\mathfrak{M}_{g,n}$}
 (0,0) node(A4){$S$}
 (1-0.1+1,1-0.1)--(1+0.1+1,1-0.1)--(1+0.1+1,1+0.1)--(1-0.1+1,1+0.1)--(1-0.1+1,1-0.1);
\draw (2,0) node[below]{$s$};
 
\draw[->,>=stealth] (A1) to[bend right=0] (A2);
\draw[->,>=stealth] (A1) to[bend right=0] (A4);
\draw[->,>=stealth] (A2) to[bend right=0] (A3);
\draw[->,>=stealth] (A4) to[bend right=0] (A3);
\end{tikzpicture}
\end{center}

\begin{lem}\label{ident}
Let $F$ denote the pull-back $F:=p^*(Q)$ by the projection $p \colon P_s \to C$ of a divisor $Q=q(S)$ given by any section $q \colon S \to C$ of the family $s$.  
Then, we have an isomorphism
$$\cD(s) \simeq
\bigsqcup_{\beta} \left[\mathcal{K}^o_{s,\beta}(s)/\mathcal{A}ut(s)\right]$$
of fibered categories over $S$,
where the disjoint union is taken over curve-classes $\beta \in H_2(P_s,\ZZ)$ satisfying
$$\beta \cdot F=1.$$
The number of such curve-classes $\beta$ is finite.
Consequently, the fibered category $\cD(s)$ is an algebraic stack of finite type over the base scheme $S$.
\end{lem}

\begin{proof}
Denote by $\mathrm{id}_S \colon S \to S$ the identity morphism of the scheme $S$.
First, by Remarks \ref{DFaut}, \ref{homol}, and \ref{logmap}, we have an isomorphism of the groupoids of $\mathrm{id}_S$-points
$$\cD(s)(\mathrm{id}_S) \simeq
\bigsqcup_{\beta} \left[\mathcal{K}^o_{s,\beta}(s) / \mathcal{A}ut(s) \right](\mathrm{id}_S).$$

Second, the construction of the log-scheme $(P_s,\cM^{0,\infty})$ is functorial, i.e.~for any morphism $u \colon U \to S$, we have a cartesian diagram
\begin{center}
\begin{tikzpicture}[scale=0.75]
\draw (0,4) node(B1){$P_{s \circ u}$}
 (2+2,4) node(B2){$P_s$}
 (0,2) node(A1){$C'$}
 (2+2,2) node(A2){$C$}
 (2+2,0) node(A3){$S$}
 (0,0) node(A4){$U$}
 (1-0.1+1,1-0.1)--(1+0.1+1,1-0.1)--(1+0.1+1,1+0.1)--(1-0.1+1,1+0.1)--(1-0.1+1,1-0.1)
 (1-0.1+1,1-0.1+2)--(1+0.1+1,1-0.1+2)--(1+0.1+1,1+0.1+2)--(1-0.1+1,1+0.1+2)--(1-0.1+1,1-0.1+2);
\draw (2,0) node[below]{$u$};

\draw[->,>=stealth] (B1) to[bend right=0] (B2);
\draw[->,>=stealth] (B1) to[bend right=0] (A1);
\draw[->,>=stealth] (B2) to[bend right=0] (A2);
\draw[->,>=stealth] (A1) to[bend right=0] (A2);
\draw[->,>=stealth] (A1) to[bend right=0] (A4);
\draw[->,>=stealth] (A2) to[bend right=0] (A3);
\draw[->,>=stealth] (A4) to[bend right=0] (A3);
\draw[->,>=stealth] (A1) to[bend left=30] (B1);
\draw[->,>=stealth] (A2) to[bend right=30] (B2);
\draw (-1,3) node{$0, \infty$};
\draw (5,3) node{$0, \infty$};
\end{tikzpicture}
\end{center}
compatible with the $0$ and the $\infty$ sections.
Thus, there is a canonical correspondence between log-maps $C' \to P_{s \circ u}$
and $C' \to P_s$, and an isomorphism
$$\cD(s) \simeq
\bigsqcup_{\beta} \left[\mathcal{K}^o_{s,\beta}(s)/\mathcal{A}ut(s)\right]$$
of fibered categories over $S$.
Note that by Remark \ref{homol}, the number of $\beta$ in the above disjoint union is finite.
Therefore, the fibered category $\cD(s)$ is an algebraic stack of finite type over the scheme $S$.
\end{proof}

Since for any morphism $s \colon S \to \mathfrak{M}_{g,n}$, the fibered product
$$\cD(s) = \cD \times_{\mathfrak{M}_{g,n}} S$$
is an algebraic stack of finite type over $S$, then the fibered category $\cD$ is an algebraic stack of finite type over $\mathfrak{M}_{g,n}$, and therefore over a point.
Moreover, by the definition \ref{stabledef} of stability, together with Lemma \ref{repres}, for any $k$-log-canonical divisor $\xi$ over a geometric point $S$, the automorphism group $\mathcal{A}ut_S(\xi)(S)$ is finite, so that the stack $\cD(s)$ is a Deligne--Mumford stack.

Forgetting all the log-structures in the above arguments, we show that the fibered category $D$ is also a Deligne--Mumford stack, and the representability of the map
$$\mathrm{FL} \colon \cD \to D$$
comes from the isomorphism of groups
$$\mathcal{A}ut_S(\xi)(S) \to \mathcal{A}ut_S(\underline{\xi})(S)$$
for every object $\xi \in \cD(S)$ over a geometric point $S$, see Lemma \ref{repres}.
The finiteness of the map $\mathrm{FL}$ comes from \cite[Proposition 3.7.5]{QChen}, i.e.~there are at most finitely many minimal $k$-log-canonical divisors over a geometric point with fixed underlying map and marked graph, together with the fact there are finitely many possible marked graphs for a given topological type, see Remark \ref{finitemarkedgraph}.

\subsection{$\cD$ is proper}
The stack $\cD$ has finite diagonal,
 so that it admits a finite surjective morphism from a scheme, by \cite[Theorem 2.7]{EHKV}.
Therefore, it is sufficient to check the weak 
valuative criterion to prove the stack is proper.

In the following lemma, we denote by $R$ a discrete valuation ring, $K$ the fraction field of $R$, $S:=\mathrm{Spec}(R)$, and $s$ and $\eta$ the closed and generic point of $S$.
Similarly, we use the notations $K'$, $S'$, $s'$, and $\eta'$ for another discrete valuation ring $R'$.

\begin{lem}\label{wvc}
Let $\xi_\eta \in \cD(\eta)$.
Up to a base change given by an injection $R \hookrightarrow R'$ of discrete valuation rings such that $K \hookrightarrow K'$ is a finite field extension, we have a cartesian diagram
\begin{center}
\begin{tikzpicture}[scale=0.6]
\draw (0,2) node(A1){$\xi_{\eta'}$}
 (2+2,2) node(A2){$\xi_{S'}$}
 (2+2,0) node(A3){$S'$}
 (0,0) node(A4){$\eta'$}
 (1-0.1+1,1-0.1)--(1+0.1+1,1-0.1)--(1+0.1+1,1+0.1)--(1-0.1+1,1+0.1)--(1-0.1+1,1-0.1);
 
\draw[->,>=stealth] (A1) to[bend right=0] (A2);
\draw[->,>=stealth] (A1) to[bend right=0] (A4);
\draw[->,>=stealth] (A2) to[bend right=0] (A3);
\draw[->,>=stealth] (A4) to[bend right=0] (A3);
\end{tikzpicture}
\end{center}
where $\xi_{\eta'}$ is the pull-back of $\xi_\eta$ via $\eta' \to \eta$ and we have $\xi_{S'} \in \cD(S')$.
Furthermore, the extension $\xi_{S'}$ is unique up to a unique isomorphism.
\end{lem}

\begin{proof}
Consider the minimal and stable $k$-log-canonical divisor
\begin{equation*}
\xi_\eta = (\pi_\eta \colon C_\eta \rightarrow \eta, (\sigma_1,\dotsc,\sigma_n), \cM_\eta^{C_\eta/\eta} \rightarrow \cM_\eta ~ ; ~~  C_\eta \xrightarrow{f_\eta} P_\eta, f_\eta^*\cM^{0,\infty} \xrightarrow{f_\eta^\flat} \cM_{C_\eta})
\end{equation*}
over the generic point $\eta$ of the scheme $S:=\mathrm{Spec}(R)$.
In particular, it satisfies the two conditions of Lemma \ref{stable}, i.e.~the only unstable components of the curves are of type $(\PP^1,0,\infty)$ and they are not contracted to a point via the schematic map $f_\eta$. Moreover, the corresponding projective bundle $P_\eta$ is the trivial $\PP^1$-bundle on these components.
Since the stacks $\overline{\cM}_{0,2}(\PP^1)$ and $\overline{\cM}_{g,n}$ are proper, possibly after a base change, we can extend $\pi_\eta$ into a family
$$\pi_S \colon C_S \to S$$
of prestable curves with no unstable components of type $(\PP^1,0)$.
We form the corresponding projective bundle $P_S$ over $C_S$ and endow it with the divisorial log-structure $\cM^{0,\infty}$ given by the sections $0$ and $\infty$. It is a Deligne--Faltings pair of rank one, projective over $S$.

The $k$-log-canonical divisor $\xi_\eta$ is in particular a minimal and stable log-map over $\eta$ to the log-scheme $P_S$, i.e.~we have
$$\xi_\eta \in \mathcal{K}(P_S)(\eta),$$
as defined and denoted in \cite{QChen}.
By \cite[Theorem 6.1.1]{QChen}, the stack $\mathcal{K}(P_S)$ is proper, so that we can extend $\xi_\eta$ as a log-map
$$\xi' \in \mathcal{K}(P_S)(S')$$
over $S'$ obtained by base-change from $S$.
The log-map $\xi'$ is thus given by the log-morphisms
\begin{center}
\begin{tikzpicture}[scale=1]
\draw (0,0) node(A1){$(C_{S'},\cM_{C_{S'}})$}
 (4,0) node(A2){$(P_{S},\cM^{0,\infty})$}
 (0,-1) node(A3){$(S',\cM_{S'})$}
 (2,0) node[above]{$f_{S}$};
 
\draw[->,>=stealth] (A1) to[bend right=0] (A2);
\draw[->,>=stealth] (A1) to[bend right=0] (A3);
\end{tikzpicture}
\end{center}
Composing with the projection morphism of schemes $p_S \colon P_S \to C_{S}$, we obtain a morphism of schemes
$$p_S \circ f_{S} \colon C_{S'} \to C_S$$
whose restriction over the base $\eta'$ is the morphism $\Phi$ of families of prestable curves induced by the base-change $\eta' \to \eta$.
Since the family of prestable curves $C_{S'} \to S'$ is flat, then we have an equality
$$p_S \circ f_{S} = \Phi$$
of the schematic morphisms.
Furthermore, up to base-change, we can consider that this family of schematic morphisms satisfies the two conditions of Lemma \ref{stable}.
From the commutative diagram of log-schemes
\begin{center}
\begin{tikzpicture}[scale=0.75]
\draw (0,2) node(A1){$(P_{S'},\cM^{0,\infty})$}
 (2+2,2) node(A2){$(P_S,\cM^{0,\infty})$}
 (2+2,0) node(A3){$(C_S,\cO^*_{C_S})$}
 (0,0) node(A4){$(C_{S'},\cO^*_{C_{S'}})$}
 (7,-1) node(B){$(C_{S'},\cM_{C_{S'}})$}
 (1-0.1+1,1-0.1)--(1+0.1+1,1-0.1)--(1+0.1+1,1+0.1)--(1-0.1+1,1+0.1)--(1-0.1+1,1-0.1);
\draw (2,0) node[below]{$\Phi$};
\draw (6,0.7) node[right]{$f_{S}$};
 
\draw[->,>=stealth] (A1) to[bend right=0] (A2);
\draw[->,>=stealth] (A1) to[bend right=0] (A4);
\draw[->,>=stealth] (A2) to[bend right=0] (A3);
\draw[->,>=stealth] (A4) to[bend right=0] (A3);
\draw[->,>=stealth] (B) to[bend right=10] (A2);
\draw[->,>=stealth] (B) to[bend left=10] (A4);
\end{tikzpicture}
\end{center}
we get a canonical log-map
\begin{center}
\begin{tikzpicture}[scale=1]
\draw (0,0) node(A1){$(C_{S'},\cM_{C_{S'}})$}
 (4,0) node(A2){$(P_{S'},\cM^{0,\infty})$}
 (0,-1) node(A3){$(S',\cM_{S'})$}
 (2,0) node[above]{$f_{S'}$};
 
\draw[->,>=stealth] (A1) to[bend right=0] (A2);
\draw[->,>=stealth] (A1) to[bend right=0] (A3);
\end{tikzpicture}
\end{center}
satisfying the two conditions of Lemma \ref{stable} and whose schematic map $f_{S'}$ is a section.
Hence, we obtain a point
$\xi_{S'} \in \cD(S').$

If we have two extensions $\xi^{(1)}_{S},\xi^{(2)}_{S} \in \cD(S)$ of the family $\xi_\eta\in \cD(\eta)$, then they are extensions as log-maps of two families $\xi^{(1)}_\eta, \xi^{(2)}_\eta \in \mathcal{K}(P_{S})(\eta)$, which are isomorphic by an automorphism $\phi_1:=(\rho_1,\theta_1)$ of families of log-curves over $\eta$.
The automorphism $\phi_1$ is extended uniquely as an automorphism between the families $\pi^{(1)}_S$ and $\pi^{(2)}_S$ over the base scheme $S$.
Thus, we get two extensions $\phi^*_1\xi^{(1)}_{S}, \xi^{(2)}_{S}$ as log-maps of $\xi^{(2)}_\eta$.
By \cite[Theorem 6.1.1]{QChen}, there is a unique isomorphism $\phi_2:=(\rho_2,\theta_2)$ between $\phi^*_1\xi^{(1)}_{S}$ and $\xi^{(2)}_{S}$, and the composition $\phi_2 \circ \phi_1$ is an isomorphism of $(\xi^{(1)}_{S},\xi^{(2)}_{S})$ as $k$-log-canonical divisors. It is unique since the automorphisms $\phi_1$ and $\phi_2$ are unique.
\end{proof}

\subsection{Perfect obstruction theory}
The following description is very similar to the construction of a perfect obstruction theory for Gromov--Witten theory, replacing the usual cotangent complex by Olsson's log-cotangent complex \cite[Definition 3.2]{Olsson}.
In this section, we closely follow \cite[Section 5]{GrossSieb}.
Note that, by \cite[Proposition 4.8]{QChen2}, the notion of a basic stable log-map from \cite[Definition 1.17]{GrossSieb} coincides with the notion of a minimal stable log-map from \cite{QChen,QChen2}.

Let $\mathfrak{M}$ denote the log-stack defined in \cite[Appendix A]{GrossSieb}, whose $S$-points are families of prestable curves over $S$, together with the choice of a fine log-structure on the base $S$.
Consider $\mathfrak{C} \to \mathfrak{M}$ the universal curve, which is also a log-stack, and $p \colon \mathfrak{P} \to \mathfrak{C}$ the associated universal projective bundle, where the log-structure is
$$\cM_{\kP}:= p^*(\cM_\kC) \oplus_{\cO^*_\kP} \cM^{0,\infty},$$
so that the projection $p \colon \mathfrak{P} \to \kC$ is a log-smooth, projective, and representable morphism of log-stacks.

The Deligne--Mumford stack $\cD$ is canonically endowed with the following log-structure.
For any morphism $g \colon T \to \cD$, we take the log-structure $g^*\cM_{\cD}:=\cM_T$ corresponding to the log-structure on the base of the $k$-log-canonical divisor $g \in \cD(T)$.
We have a strict log-morphism of log-stacks $\cD \to \kM$ sending a $k$-log-canonical divisor $\xi$ over $S$ to its family of prestable log-curves $(C,\cM_C) \to (S,\cM_S)$.
Then, we consider the universal curve
$$\pi \colon  \cC \to \cD \quad \textrm{with} \quad \cC := \cD \times_{\kM} \kC,$$
and the evaluation morphism
$$f \colon \cC \to \kP;$$
they are log-morphisms of log-stacks.
Following \cite[Section 5]{GrossSieb} and \cite[Remark 5.2]{GrossSieb}, the commutative diagram
\begin{center}
\begin{tikzpicture}[scale=0.7]
\draw (0,2) node(A1){$\cC$}
 (2,2) node(A2){$\kP$}
 (2,0) node(A3){$\kC$}
 (0,0) node(A4){$\kC$}
 (1,1) node{$\circlearrowright$}
 (1,2) node[above]{$f$}
 (1,0) node[below]{$\mathrm{id}_{\kC}$};
 
\draw[->,>=stealth] (A1) to[bend right=0] (A2);
\draw[->,>=stealth] (A1) to[bend right=0] (A4);
\draw[->,>=stealth] (A2) to[bend right=0] (A3);
\draw[->,>=stealth] (A4) to[bend right=0] (A3);
\end{tikzpicture}
\end{center}
yields a morphism of log-cotangent complexes
$$Lf^* L^\bullet_{\kP/\kC} \to L^\bullet_{\cC/\kC},$$
which induces a morphism
$$\phi \colon R\pi_*(Lf^*L^\bullet_{\kP/\kC} \otimes^L \omega_{\underline{\pi}}) \to L^\bullet_{\underline{\cD}/\underline{\kM}},$$
where we underline the objects considered without log-structure.
By log-smoothness of the log-morphism $\kP \to \kC$, the object $L^\bullet_{\kP/\kC}$ is represented by a locally free sheaf defined as follows.
For any morphism $g \colon T \to \kP$, we take the locally free sheaf $$\Omega^1_{(P_g,\cM_{P_g}) / (C_{p \circ g},\cM_{C_{p \circ g}})}$$ of log-differentials of the log-smooth morphism $(P_g,\cM_{P_g}) \to (C_{p \circ g},\cM_{C_{p \circ g}})$.
We denote the locally free sheaf on $\kP$ by $\Omega^1_{\kP/\kC}$.
Hence, we get a morphism
$$\phi \colon E^\bullet \to L^\bullet_{\underline{\cD}/\underline{\kM}},$$
where the object
$$E^\bullet := (R\pi_* \left[ f^* \Theta_{\kP/\kC} \right])^\vee$$
is of perfect amplitude contained in $\left[ 0,1 \right]$, and $\Theta_{\kP/\kC}$ is the dual of $\Omega^1_{\kP/\kC}$.

Following \cite[Proposition 5.1]{GrossSieb}, we claim that the morphism $\phi$ is a perfect obstruction theory relative to the stack $\underline{\kM}$.
The proof is exactly the same as in \cite{GrossSieb}. We reproduce the main arguments.

Let $T \to \overline{T}$ be a square-zero extension of schemes over $\underline{\kM}$, with ideal $\mathcal{J}$, and let $\xi_T \in \underline{\cD}(T)$. To simplify notations, we write $g \colon T \to \underline{\cD}$ instead of $\xi_T$.
Note that the maps from $T$ and $\overline{T}$ to $\underline{\kM}$ give log-structures on $T$ and on $\overline{T}$.
We have a commutative diagram of log-stacks
\begin{center}
\begin{tikzpicture}[scale=1]
\draw (0,2) node(A1){$\cC_T$}
 (2,2) node(A2){$\cC$}
 (2,0) node(A3){$\kC$}
 (0,0) node(A4){$\cC_{\overline{T}}$}
 (3,2) node(A5){$\kP$}
 (3,0) node(A6){$\kC$}
 (-1+0,-1+2) node(B1){$T$}
 (-1+2,-1+2) node(B2){$\cD$}
 (-1+2,-1+0) node(B3){$\kM$}
 (-1+0,-1+0) node(B4){$\overline{T}$};
 
\draw[->,>=stealth] (A1) to[bend right=0] (A2);
\draw[->,>=stealth] (A1) to[bend right=0] (A4);
\draw[->,>=stealth] (A2) to[bend right=0] (A3);
\draw[->,>=stealth] (A4) to[bend right=0] (A3);
\draw[->,>=stealth] (B1) to[bend right=0] (B2);
\draw[->,>=stealth] (B1) to[bend right=0] (B4);
\draw[->,>=stealth] (B2) to[bend right=0] (B3);
\draw[->,>=stealth] (B4) to[bend right=0] (B3);
\draw[->,>=stealth] (A1) to[bend right=0] (B1);
\draw[->,>=stealth] (A2) to[bend right=0] (B2);
\draw[->,>=stealth] (A3) to[bend right=0] (B3);
\draw[->,>=stealth] (A4) to[bend right=0] (B4);
\draw[->,>=stealth] (A2) to[bend right=0] (A5);
\draw[->,>=stealth] (A3) to[bend right=0] (A6);
\draw[->,>=stealth] (A5) to[bend right=0] (A6);
\end{tikzpicture}
\end{center}
where all but the front and back faces of the cube are cartesian. Denote by $\widetilde{g}$ the map from $\cC_T$ to $\cC$.
We need to study the extension problem of $g$ to $\overline{T}$.
Recall from Remark \ref{DFaut} that the automorphism group $\mathcal{A}ut_S(\xi)$ of a $k$-log-canonical divisor $\xi$ over $S$ is contained inside the automorphism group $\mathcal{A}ut_S(C)$ of the underlying family of curves over $S$. It is still true when we consider families of log-curves over $S$. It means that the log-morphism $\cD \to \kM$ is representable. Moreover, it preserves the log-structures, i.e.~it is a strict log-morphism.
Therefore, we can use the standard obstruction theory for the extension problem, see \cite[Chapter 3]{Ill} and \cite{Olsson3}, and the obstruction class $\omega(g) \in \mathrm{Ext}^1_{\cO_T}(Lg^* L^\bullet_{\underline{\cD}/\underline{\kM}},\mathcal{J})$ is given by
$$Lg^* L^\bullet_{\underline{\cD}/\underline{\kM}} \to L^\bullet_{T/\overline{T}} \to \tau_{\geq -1} L^\bullet_{T/\overline{T}} = \mathcal{J}\left[1\right].$$
On the other hand, an extension of $g$ to $\overline{T}$ exists if and only if the universal log-morphism $(f_T,f^\flat_T) \colon \cC_T \to \kP$ extends as a log-morphism to $\cC_{\overline{T}}$, because of the homological description of a log-section.
By \cite[Theorem 8.45]{Olsson}, the obstruction class in this case is $o \in \mathrm{Ext}^1_{\cO_{\cC_T}}(Lf_T^* L^\bullet_{\kP/\kC},\pi_T^*\mathcal{J})$ given by
$$Lf_T^* L^\bullet_{\kP/\kC} \to L^\bullet_{\cC_T/\cC_{\overline{T}}} \to \tau_{\geq -1} L^\bullet_{\cC_T/\cC_{\overline{T}}} = \pi_T^* \mathcal{J}\left[ 1 \right].$$
As in the proof of \cite[Proposition 5.1]{GrossSieb}, we show that
\begin{eqnarray*}
\mathrm{Ext}^l_{\cO_T}(Lg^* E^\bullet,\mathcal{J}) & = & \mathrm{Ext}^l_{\cO_T}(Lg^* R\pi_*(Lf^* L^\bullet_{\kP/\kC} \otimes^L \omega_\pi), \mathcal{J}) \\
& = & \mathrm{Ext}^l_{\cO_\cC}(Lf^* L^\bullet_{\kP/\kC} \otimes^L \omega_\pi,L\pi^!Rg_*\mathcal{J}) \\
& = & \mathrm{Ext}^l_{\cO_\cC}(
Lf^* L^\bullet_{\kP/\kC},L\pi^!R\widetilde{g}_*p^*\mathcal{J}
) \\
& = & \mathrm{Ext}^l_{\cO_{\cC_T}}(Lf_T^* L^\bullet_{\kP/\kC},\pi_T^*\mathcal{J})
\end{eqnarray*}
for any $l \in \NN$, and for $l=1$, the element $\phi^* \omega(g)$ is sent to the above obstruction class $o$, proving the obstruction part of the criterion in \cite[Theorem 5.3.3]{BF}. For $l=0$, we get the torsor part of \cite[Theorem 5.3.3]{BF}.

\section{Relation to twisted canonical divisors}
In this section, we study the relation between $k$-log-canonical divisors and twisted canonical divisors as introduced in \cite[Definition 20]{PandhFark}.
In particular, we will exhibit a closed substack $\cD^{*(k)}$ of $\cD^{(k)}$ with a morphism to the moduli space $\widetilde{\mathcal{H}}^k_g(\vec{\mathbf{m}})$ of twisted canonical divisors with multiplicities $\vec{\mathbf{m}}=(\mathbf{m}_1,\dotsc,\mathbf{m}_n)$ at the markings.
Furthermore, the closed substack $\cD^{*(k)}$ is also equipped with a perfect obstruction theory and a virtual fundamental cycle.
In the last section, we isolate the so-called principal component of the stack $\cD^{*(k)}$, and in the strictly meromorphic case with $k=1$, we prove the relation between its virtual fundamental cycle and the weighted fundamental cycle from \cite[Appendix A.4]{PandhFark}.

\subsection{Description of a $k$-log-canonical divisor}\label{descr}
We study further a minimal and stable $k$-log-canonical divisor
\begin{equation*}
\xi = (\pi \colon C \rightarrow S, (\sigma_1,\dotsc,\sigma_n), \cM_S^{C/S} \rightarrow \cM_S ~ ; ~~  C \xrightarrow{f} P, f^*\cM^{0,\infty} \xrightarrow{f^\flat} \cM_C)
\end{equation*}
over a geometric point $S=\mathrm{Spec}(\CC)$, with topological type $\Gamma=(g,n,\vec{\mathbf{c}})$ and marked graph $G_\xi$.
Take the partial normalization $\nu \colon \widetilde{C} \to C$
obtained by normalizing exactly the nodes $p_1,\dotsc,p_m$ with non-zero contact orders.
Thus, the curve $\widetilde{C}$ can be disconnected and have singular fibers, but the degeneracies and the signs are constant on each of its connected components.
We denote by $p'_k$ and $p''_k$ the two preimages of a normalized node $p_k$ under the map $\nu$, and we choose the convention that $p'_k$ gets the positive contact order and $p''_k$ gets the negative contact order.

Consider all the markings $\Omega:=\left\lbrace \sigma_i \right\rbrace \cup \left\lbrace p'_k \right\rbrace \cup \left\lbrace p''_k \right\rbrace$ on the curve $\widetilde{C}$.
We form the divisors
$$D_0 := \sum_{x \in \Omega^+} c_x\cdot [x] \quad \textrm{and} \quad D_\infty := \sum_{x \in \Omega^-} c_x \cdot [x],$$
where $\Omega^+$ (resp.~$\Omega^-$) is the set of markings with sign $+$ (resp.~$-$), the integer $c_x$ is the contact order at the marking $x$. These divisors are completely determined by the marked graph $G_\xi$.
Note that a marking contributes to the divisor $D_0$ (resp.~to $D_\infty$) if it is sent to $0$ (resp.~to $\infty$) by the schematic map $f \circ \nu$.
Note also the relation $\omega_{\widetilde{C}}^k = \nu^* \omega_{C}^k$.

\begin{pro}\label{detail}
A minimal and stable $k$-log-canonical divisor $\xi$ with marked graph $G_\xi$ is equivalent to a non-vanishing global section
$$s \in H^0(\widetilde{C},\omega_{\widetilde{C}}^k(D_\infty - D_0)),$$
where the curve $\widetilde{C}$ is the partial normalization of the curve $C$ at the nodes with non-zero contact orders, and the divisors $D_0$ and $D_\infty$ are uniquely determined by the marked graph $G_\xi$.
Up to isomorphisms, the restriction of the global section to degenerate components is unique.
\end{pro}

In particular, a necessary condition for the existence of a $k$-log-canonical divisor with a given marked graph is the compatibility
\begin{equation}\label{compat}
\sum_{h \to v} \mathbf{c}_v = k(2 g_v - 2 + n_v)
\end{equation}
for each vertex $v$, where the sum is taken over all the half-edges $h$ incident to $v$.

\begin{proof}
We start with some comments about the log-structures.
First, observe that, by the minimality condition, the monoid $\overline{\cM}_S$ is completely determined by the marked graph $G_\xi$. It is a discrete monoid, and it is generated by the degeneracies $e_{C_v}$ of the irreducible components $C_v$ of $C$ and by the elements $e_p$ smoothing the nodes $p \in C$.
We fix representatives of these elements in the monoid $\cM_S$.

Secondly, on a sufficiently small \'etale neighborhood of a point $x \in C$ which is not a marking and not a node, the monoid $\cM_C(U)$ is generated by the invertible functions $\cO^*_C(U)$ and by the elements of $\overline{\cM}_S$.
If $x$ is a marking, we need to add in the generators any local coordinate $u$ centered at $x$.
If $x$ is a node, we need to add in the generators any local coordinates $u,v$ of the two branches. We can choose these local coordinates such that the relation $e_x = u \cdot v \in \cM_C$ holds, where $e_x \in \cM_S$ is the element smoothing the node.

Thirdly, on a sufficiently small \'etale neighborhood of a point $x \in C$ such that $f(x) \notin \{ 0 , \infty \}$, the monoid $f^*\cM^{0,\infty}$ is equal to $\cO^*_C(U)$.
If $f(x)=0$, then we need to add in the generators any local coordinate $w$ on a neighborhood of $f(U)$ such that the image of $U$ by the zero section has equation $0(U)= \left\lbrace w=0 \right\rbrace$.
In particular, it is also a local coordinate of the line bundle $\omega_C^k$.
If $f(x)=\infty$, then we need to add in the generators any local coordinate $w$ on a neighborhood of $f(U)$ such that the image of $U$ by the infinity section has equation $\infty(U)= \left\lbrace w=0 \right\rbrace$.
In particular, it is also a local coordinate of the line bundle $\omega_C^{-k}$.

Fourth, the morphism of log-structures $f^\flat \colon f^*\cM^{0,\infty} \to \cM_C$ is always trivially determined on the subsheaf $\cO_C^* \subset f^*\cM^{0,\infty}$, because it coincides with the embedding $\cO^*_C \subset \cM_C$.
Therefore, we see that on any \'etale open subset $U \subset C$ such that $f(U)$ intersects the sections $0$ and $\infty$ on a finite set of points, the morphism $f^\flat$ is uniquely determined by the schematic map $f$.
In particular, it is the case on the complement in $C$ of the images via $\nu$ of all the degenerate connected components of the curve $\widetilde{C}$.

Let us consider a connected component $\widetilde{C}_1$ of the curve $\widetilde{C}$ on which the degeneracy is zero.
Then for any non-marked point $x \in \widetilde{C}_1$, we have
$f(\nu(x)) \notin \{ 0 , \infty \},$
and at every marked point of the curve $\widetilde{C}_1$, the multiplicity of the intersection of $f(\nu(\widetilde{C}_1))$ with the divisors $\left[0\right]$ and $\left[\infty\right]$ is given by the contact order.
Thus it is clear that the schematic map $f$ restricted to $\nu(\widetilde{C_1})$ is equivalent to a non-vanishing global section
$$s \in H^0(\widetilde{C}_1,\omega_{\widetilde{C}_1}^k(D_\infty - D_0)).$$
We have seen that, in the non-degenerate case, all the information is contained in the schematic map.

On a degenerate component $C_2$, on the contrary, the schematic map $f$ is trivial, since it is equal to the section $0$ or $\infty$.
Therefore, all the information is contained in the morphism of log-structures $f^\flat \colon f^*\cM^{0,\infty} \to \cM_C$.
Let us treat the case where $f(C_2)=0$, the other case is similar.
We follow the proof of \cite[Proposition 5.2.4]{QChen} and we use the notations of the comments in the beginning of this proof.
Fix a point $x \in C_2$ and a sufficiently small \'etale neighborhood $U$ of $x \in C$, and consider the three following situations:
\begin{itemize}
\item Assume that $x$ is not a marking and not a node. We have
$$f^\flat(w) = e_{C_2} \cdot h \in \pi^*\cM_S(U),$$ where $e_{C_2}$ is the degeneracy of the component $C_2$ and $h \in \cO_C^*(U)$.
\item Assume that $x$ is a marking. For any local coordinate $u$ centered at $x$, we have
$$f^\flat(w)=e_{C_2} \cdot u^{c_x} \cdot h \in \cM_C(U),$$ where $c_x$ is the contact order at the marking and $h \in \cO_C^*(U)$.
\item Assume that $x$ is a node and denote by $C_3$ the other component attached to $x$. We have local coordinates $u,v$, with $u$ a local coordinate on $C_2$ near $x$ and $v$ is a local coordinate on $C_3$ near $x$, such that
$$e_x = u \cdot v \in \cM_C(U) \quad \textrm{and} \quad f^\flat(w)= \left\lbrace
\begin{array}{ll}
e_{C_2} \cdot u^{c_x} \cdot h, & \textrm{if $c_x \geq 0$,} \\
e_{C_3} \cdot v^{-c_x} \cdot h, & \textrm{if $c_x \leq 0$,}
\end{array}
\right.
$$ where $e_x \in \pi^*\cM_S$ is the element smoothing the node, $e_{C_2}, e_{C_3}$ are the degeneracies of the components $C_2, C_3$, $c_x \in \ZZ$ is the contact order at the node on the branch corresponding to $C_2$, and $h \in \cO_C^*(U)$.
By the minimality condition, we have
\begin{equation}\label{freechoice}
e_{C_3}= \lambda \cdot e_{C_2} \cdot e_x^{c_x} \in \cM_S \quad \textrm{or} \quad e_{C_2}=\lambda \cdot e_{C_3} \cdot e_x^{c_x} \in \cM_S,
\end{equation}
depending on the sign of the contact order $c_x$, with $\lambda \in \CC^*$.
\end{itemize}
Therefore, we can interpret the restriction of the morphism of log-structures $f^\flat$ to the component $C_2$ as a cosection of the line bundle $\omega_{C_2}^k$ vanishing or having poles at the markings and the nodes of the component $C_2$ with prescribed orders given by the contact orders.
It is then a non-vanishing section of $\omega_{C_2}^{-k}(D_0)$.
Note that by the free choice of the constant $\lambda \in \CC^*$ in the equation \eqref{freechoice}, no compatibility condition is asked at the nodes.
By an automorphism of the $k$-log-canonical divisor $\xi$, we are free to rescale every non-zero degeneracy independently, so that we can rescale the section $s \in H^0(\widetilde{C},\omega_{\widetilde{C}}^k(D_\infty - D_0))$ corresponding to $\xi$ independently on each degenerate connected component of the curve $\widetilde{C}$.
\end{proof}

\subsection{Rubber part of the moduli space}\label{rubber}
In Proposition \ref{detail}, we have seen that a $k$-log-canonical divisor $\xi$ over a geometric point $S$ with a given marked graph $G_\xi$ is uniquely determined on the degenerate components of the source curve, since the automorphism group $\mathcal{A}ut_S(\xi)(S)$ acts on $\xi$ by rescaling the morphism $f^\flat$ independently on each degenerate component.
Nevertheless, we still have some freedom on the non-degenerate components, since the schematic map $f$ is not fixed.
Here, the idea is to rescale as well the $k$-log-canonical divisor $\xi$ on the non-degenerate components, independently on each component. As a consequence, such a $k$-log-canonical divisor will be uniquely determined by its marked graph.

\begin{dfn}\label{Cstaradm}
A $k$-log-canonical divisor $\xi$ over a scheme $S$ is $\CC^*$-admissible if, for every geometric point $s \in S$, the induced schematic morphism
$$f_s \colon C_s \to P_s$$
is constant equal to the section $0$ or $\infty$ on each stable irreducible component of the curve $C_s$.
When a $k$-log-canonical divisor is $\CC^*$-admissible, minimal, and stable, we call it a rubber $k$-canonical divisor.
\end{dfn}

\begin{rem}
Take a $k$-log-canonical divisor $\xi$ over a geometric point and assume that there is a stable irreducible component $C_v$ of the source curve $C$ where the schematic map $f$ is constant, but is not equal to the sections $0$ or $\infty$ on $C_v$.
Then the component $C_v$ is not degenerate and the contact orders at its special points are all zero.
By the compatibility condition \eqref{compat} and by the stability of $C_v$, we deduce that the integer $k$ has to be zero.
\end{rem}

\begin{rem}\label{P1}
For a stable $k$-log-canonical divisor $\xi$ over a scheme $S$, the only unstable components of the source curve $C_s$ of a geometric fiber over $s \in S$ are of type $(\PP^1,0,\infty)$, and that the schematic map $f_s$ is then non-constant, see Lemma \ref{stable}.
Precisely, on these unstable components, the schematic map is of the form
\begin{equation*}
\begin{array}{rccc}
g_\lambda \colon & \PP^1 & \to & \PP^1 \times \PP^1 \\
& \left[ x : y \right] & \mapsto & \left[ x : y \right] , \left[ x^c : \lambda y^c \right]
\end{array}
\end{equation*}
where $c \in \NN$ is the contact order at the two special points and $\lambda \in \CC^*$.
Moreover, there are exactly $c$ distinct automorphisms
$\Phi \colon \left[ x : y \right] \mapsto \left[ x : \mu y \right]$
of $(\PP^1,0,\infty)$ sending $g_\lambda$ to $g_1$, choosing a $c$-rooth $\mu$ of $\lambda$.
Therefore, these non-degenerate unstable components are naturally and independently rescaled by elements of $\mathcal{A}ut_S(\xi)(S)$.
Furthermore, we naturally associate to $\xi$ a stable map to the projective line $\PP^1$, i.e.~an element in $\overline{\cM}_{g,n}(\PP^1,d)$, where $d$ is the degree of the stable map. The stable map is invariant under the natural $\CC^*$-action on $\overline{\cM}_{g,n}(\PP^1,d)$. Observe also that the integer $d$ corresponds to the sum of the contact orders over all the unstable components, and it is constant in families.
\end{rem}

\begin{cor}\label{detailCstar}
Up to automorphisms, a rubber $k$-log-canonical divisor $\xi$ is uniquely determined by its marked graph $G_\xi$ and is described by any non-vanishing global section
$$s \in H^0(\widetilde{C},\omega_{\widetilde{C}}^k(D_\infty - D_0)),$$
with the notations of Proposition \ref{detail}.
\end{cor}

The schematic map $f$ of a rubber $k$-log-canonical divisor $\xi$ over a geometric point alternates between the sections $0$ and $\infty$ as follows, where the horizontal curves are all stable and the vertical lines are of the type $(\PP^1,0,\infty)$.

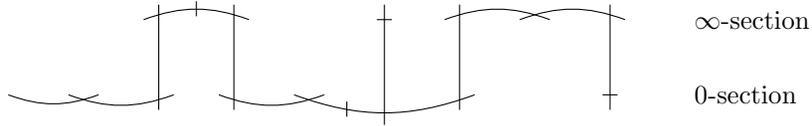
\begin{center}
\begin{tikzpicture}[scale=1]
\coordinate (A1) at (0,0);
\coordinate (A2) at (1+0.2,0);
\coordinate (A3) at (1-0.2,0);
\coordinate (A4) at (2+0.2,0);
\coordinate (A5) at (2,-0.2);
\coordinate (A6) at (2,1+0.2);
\coordinate (A7) at (2-0.2,1);
\coordinate (A8) at (3+0.2,1);
\coordinate (A9) at (3,1+0.2);
\coordinate (A10) at (3,-0.2);
\coordinate (A11) at (3-0.2,0);
\coordinate (A12) at (4+0.2,0);
\coordinate (A13) at (4-0.2,0);
\coordinate (A14) at (6+0.2,0);
\coordinate (A15) at (5,-0.4);
\coordinate (A16) at (5,1+0.2);
\coordinate (A17) at (6-0.2,1);
\coordinate (A18) at (7+0.2,1);
\coordinate (A19) at (7-0.2,1);
\coordinate (A20) at (8+0.2,1);
\coordinate (A21) at (6,-0.2);
\coordinate (A22) at (6,1+0.2);
\coordinate (A23) at (8,1+0.2);
\coordinate (A24) at (8,-0.2);

\draw[-] (A1) to[bend right=20] (A2);
\draw[-] (A3) to[bend right=20] (A4);
\draw[-] (A5) to[bend right=0] (A6);
\draw[-] (A7) to[bend left=20] (A8);
\draw[-] (A9) to[bend right=0] (A10);
\draw[-] (A11) to[bend right=20] (A12);
\draw[-] (A13) to[bend right=20] (A14);
\draw[-] (A15) to[bend right=0] (A16);
\draw[-] (A17) to[bend left=20] (A18);
\draw[-] (A19) to[bend left=20] (A20);
\draw[-] (A21) to[bend left=0] (A22);
\draw[-] (5-0.1,1) to[bend left=0] (5+0.1,1);
\draw[-] (A23) to[bend left=0] (A24);
\draw[-] (8-0.1,0) to[bend left=0] (8+0.1,0);

\draw[-] (2.5,1+0.145+0.1) to[bend left=0] (2.5,1+0.145-0.1);
\draw[-] (4.5,0-0.19+0.1) to[bend left=0] (4.5,0-0.19-0.1);

\draw (9,0) node[right]{$0$-section};
\draw (9,1) node[right]{$\infty$-section};
\end{tikzpicture}
\end{center}

\begin{pro}\label{closedcond}
The $\CC^*$-admissibility condition is a closed condition.
Precisely, let $\xi \in \cD(S)$ over a scheme $S$ such that $\xi_s$ is not $\CC^*$-admissible for some geometric point $s \in S$.
Then there is an \'etale neighborhood of $s \in S$ where none of the geometric fibers are $\CC^*$-admissible.
\end{pro}

\begin{proof}
The $\CC^*$-admissibility condition is a condition on the schemes.
Since the $k$-log-canonical divisor $\xi_s$ is not $\CC^*$-admissible, then there is a stable irreducible component $C^0_s \subset C_s$ with sign $0$, i.e.~where the schematic map is not constant equal to the section $0$ or $\infty$.

On the base scheme $S$, we have a family of curves $\pi \colon C \to S$ and a schematic map $f \colon C \to P$.
Take a point $x \in C^0_s \subset C$ away from the nodes, a local section $\sigma \colon W \to C$ of $\pi$ on an \'etale neighborhood $W$ of $s \in S$ such that $\sigma(s)=x$, and a sufficiently small \'etale neighborhood $U$ of $x \in C$ containing no nodes of the curves. Shrinking $W$, we may assume that $W = \sigma^{-1}(U)$.
Thus, for any geometric point $t \in W$, the \'etale open subset $U_t \subset C_t$ lies entirely inside an irreducible component, that we denote by $C^0_t$.

Since $f(x) \notin \left\lbrace 0(x),\infty(x) \right\rbrace$, we may assume that the image $f(U)$ is disjoint from $0(U)$ and from $\infty(U)$.
Therefore, for any geometric point $t \in W$, the schematic map $f_t$ is not constant equal to the sections $0$ or $\infty$ on the irreducible component $C^0_t \subset C_t$.
To conclude, we need to prove that the component $C^0_t$ is stable, i.e.~not of type $(\PP^1,0,\infty)$.
Indeed, the unstable components of type $(\PP^1,0,\infty)$ form a closed subspace $C^\mathrm{un} \subset C$, and since the component $C^0_s$ is stable, we have $x \notin C^\mathrm{un}$, so that we can assume $C^\mathrm{un} \cap U = \emptyset$.
\end{proof}

\begin{rem}
The $\CC^*$-admissibility condition is not an open condition.
As an example, take $k=0$ and consider the family over $\CC^*$ of usual maps to $\PP^1$ given by $(\PP^1,0,1,\infty;f_\lambda)$ with
$f_\lambda \colon \left[ x : y \right] \to \left[ x : \lambda y \right]$.
It is isomorphic to the family $(\PP^1,0,\lambda^{-1},\infty,f_1)$, and the limit $\lambda \to 0$ is the unstable curve
$$(\PP^1,0,1,\infty) \cup_{0 \sim \infty} (\PP^1,0,\infty)$$
together with the map to $\PP^1$, equal to $\infty$ on the stable component and given by the identity on the unstable component.
Therefore, the limit $\lambda =0$ is $\CC^*$-admissible, but it is false for $\lambda \neq 0$.
\end{rem}

\begin{dfn}
For any topological type $\Gamma=(g,n,\vec{\mathbf{c}})$, the moduli space $\cD^{* (k)}_\Gamma$ (resp.~$D^{* (k)}_\Gamma$) is the closed substack of the moduli space $\cD^{(k)}_\Gamma$ (resp.~ $D^{(k)}_\Gamma$) whose $S$-points are $\CC^*$-admissible.
\end{dfn}

\begin{rem}
We have several natural morphisms
$$\mathcal{D}_\Gamma^{* (k)} \to D_\Gamma^{* (k)} \hookrightarrow \overline{\cM}_{g,n}(\PP^1) \to \overline{\cM}_{g,n}$$
obtained by forgetting the log-structures, then using Remark \ref{P1}, and eventually forgetting the maps and stabilizing the source curve.
\end{rem}

\begin{dfn}
For the underlying stable curve of a stable and $\CC^*$-admissible $k$-log-canonical divisor, we say that a node or a marking is twisted of order $c$ if it is obtained after contracting an unstable component with contact orders $c$.
By convention, the twist of all other nodes and markings is zero.
Observe that a stable curve with twists at some markings and nodes have a natural structure of orbifold stable curves.
\end{dfn}

\begin{rem}\label{twisting}
The twisting order of a marking equals the contact order of the corresponding marking in the $k$-log-canonical divisor, so that it is constant in families.
Moreover, since the degree of the stable map to $\PP^1$ equals the sum of all the twisting orders and is a topological invariant of the stable map, then the sum of the twisting orders of the nodes is also constant in families.
In particular, if a stable and $\CC^*$-admissible $k$-log-canonical divisor $\xi$ over a scheme $S$ contains a stable source curve, then every curve of the family is stable.
More generally, let $\xi$ be a $k$-log-canonical divisor over a base scheme $S$, and $s \in S$ be a geometric point such that $C_s$ has a node $p$.
By \cite[Lemma 3.2.9]{QChen}, there is an \'etale open neighborhood of $s$ such that over each geometric fiber we have that either the node $p$ is smoothed out, or its contact order remains the same.
As a consequence, the twisted nodes are never smoothed out in the moduli space $\cD^{* (k)}_\Gamma$.
\end{rem}

\begin{dfn}\label{topoltwist}
The dual graph of a generic point in $\cD^{* (k)}_\Gamma$ is enhanced with positive integers at the unstable vertices corresponding to the twist orders.
We sometimes add this extra information to the topological data $\Gamma$.
\end{dfn}

\begin{rem}
It is possible to reconstruct uniquely an element in $D^{* (k)}_\Gamma$ from its underlying family of stabilized curves, and its topological type $\Gamma$.
Moreover, given a $S$-point in $\cD^{* (k)}_\Gamma$, we can normalize all the sections corresponding to twisted markings and nodes, and then we obtain $S$-points of moduli spaces $\cD^{* (k)}_{\Gamma_j}$ with different topological data $\Gamma_j$ and such that the source curve of the generic element is a smooth curve.
Hence, we have natural isomorphisms of the form
\begin{equation}\label{decom}
D^{* (k)}_\Gamma \simeq \prod_j D^{* (k)}_{\Gamma_j}.
\end{equation}
\end{rem}

We introduce the algebraic log stack constructed in \cite{Olsson2}
$$\cM := \mathcal{L}og^{\mathrm{fs}}_{(\overline{\cM}_{g,n},\cM_{\overline{\cM}_{g,n}}^\NN)},$$
where the log-structure on the moduli space of stable curves is
$$\cM^\NN_{\overline{\cM}_{g,n}}:=\cM_{\overline{\cM}_{g,n}} \oplus_{\cO^*_{\overline{\cM}_{g,n}}} \NN.$$
Therefore, an $S$-point of the log-stack $\cM$ is a family of stable curves $\pi \colon C \to S$ together with the choice of a fine and saturated log-structure $\cM_S$ on the base scheme $S$ and a morphism $\cM_S^{C/S} \oplus_{\cO^*_S} \NN \to \cM_S$; and morphisms in this category are strict.
The underlying stack is denoted by $\underline{\cM}$.

\begin{cor}\label{DMstar}
The moduli space $\mathcal{D}_\Gamma^{* (k)}$ is a proper Deligne--Mumford stack and the natural map $\mathrm{FL} \colon \mathcal{D}_\Gamma^{* (k)} \to D_\Gamma^{* (k)}$, forgetting the log-structures, is representable and finite.
Moreover, there exists a perfect obstruction theory on the moduli space $\cD_\Gamma^{*(k)}$, relative to the algebraic stack $\underline{\cM}$, and the moduli space $\cD_\Gamma^{*(k)}$ carries a virtual fundamental cycle $\left[\cD_\Gamma^{* (k)}\right]^\mathrm{vir} \in A_{\mathrm{vdim}}(\cD_\Gamma^{* (k)}),$
of dimension
$\mathrm{vdim} = 2g-2+n - \# V(\Gamma),$
where $\#V(\Gamma)$ denotes the number of vertices in the stabilized graph of the topological type $\Gamma$.
Furthermore, in genus $0$, the moduli space is log-smooth.
\end{cor}

\begin{proof}
The first part of the statement is a straightforward consequence of Theorems \ref{DMstack} and \ref{closedcond}.
For the second part, we follow the construction of the perfect obstruction theory on the stack $\mathcal{D}_\Gamma^{(k)}$, but we need to adapt it since the stack $\mathcal{D}_\Gamma^{* (k)}$ is not open in the stack $\mathcal{D}_\Gamma^{(k)}$.
By the equation \eqref{decom}, we assume that the generic source curve is smooth. The virtual fundamental class is then a product under the decomposition \eqref{decom}.

Let us first assume that every geometric source curve is of type $(\PP^1,0,\infty)$.
The schematic map $f$ is not constant and is given by Remark \ref{P1}, so that, by Proposition \ref{detail}, the morphism of log-structure $f^\flat$ is completely determined as well.
The moduli space corresponding to this case is isomorphic to the orbifold point $B\mu_c := \left[\mathrm{pt}/\mu_c \right]$, where $\mu_c$ is the cyclic group of order $c$.
We take the trivial perfect obstruction theory in that case, leading to the fundamental class of the stack $B\mu_c$.

Then we assume that the generic source curve is a smooth stable curve, so that every source curve is stable.
We may also assume that the schematic map goes to the zero section; the other case is similar.
Here, we keep the notations $\kC$, $\cC$, and $\kP$ of the proof of Theorem \ref{DMstack} for the universal objects, but considered over the log-stack $\cM$ instead of the log-stack $\kM$.
We have a commutative diagram
\begin{center}
\begin{tikzpicture}[scale=0.7]
\draw (0,2) node(A1){$\cC$}
 (2,2) node(A2){$\kC^0$}
 (2.1,1.95) node[right]{$\subset \kP$}
 (2,0) node(A3){$\kC$}
 (0,0) node(A4){$\kC$}
 (1,1) node{$\circlearrowright$}
 (1,2) node[above]{$f$}
 (1,0) node[below]{$\mathrm{id}_{\kC}$};
 
\draw[->,>=stealth] (A1) to[bend right=0] (A2);
\draw[->,>=stealth] (A1) to[bend right=0] (A4);
\draw[->,>=stealth] (A2) to[bend right=0] (A3);
\draw[->,>=stealth] (A4) to[bend right=0] (A3);
\end{tikzpicture}
\end{center}
where the log-curve $\kC^0$ denotes the universal curve $\kC$ endowed with the log-structure
$$\cM_{\kC} \oplus_{\cO^*_\kC} 0^*\cM^{0,\infty},$$
and considered as embedded in the projective bundle $\kP$.
Since the morphism $\kC^0 \to \kC$ is log-smooth, the object $L^\bullet_{\kC^0/\kC}$ is represented by a locally free sheaf $\Omega^1_{\kC^0/\kC}$ whose dual is denoted by $\Theta_{\kC^0/\kC}$.
The same proof as for Theorem \ref{DMstack} produces a perfect obstruction theory
$$\phi \colon (R\pi_* \left[ 0^*\Theta_{\kC^0/\kC} \right] )^\vee\to L^\bullet_{\underline{\cD}^*/\underline{\cM}},$$
relative to the stack of stable log-curves.
Moreover, the object $R\pi_* \left[ 0^* \Theta_{\kC^0/\kC} \right]$ has global resolutions, so that we obtain a virtual fundamental class.
We study further the relative obstruction theory in this case, to get the virtual dimension of the moduli space.

A morphism $g \colon T \to \kC^0$ corresponds to a global section $$g \colon (T,\cM_T) \to (C,\cM_C^0) ~~,~~~\cM_C^0:=\cM_C\oplus_{\cO^*_C} 0^*\cM^{0,\infty},$$ of the family of stable log-curves $\pi \colon (C,\cM_C^0) \to (T,\cM_T)$, and by the description in \cite[Proposition 3.4]{ACbook}, the sheaf of logarithmic differentials
$\Omega^1_{(C,\cM_C^0)/(C,\cM_C)}$
is the trivial line bundle $\cO_C$.
Thus, we obtain
$$\left(R\pi_* \left[ 0^* \Theta_{\kC^{0}/\kC} \right]\right)^\vee
\simeq 
\left[\cO_{\cD^*} \to \mathbb{E}^\vee \right]^\vee,$$
where the vector bundle $\mathbb{E}:=\pi_* \underline{\omega}_\pi$ is the (pull-back to $\cD^*$ of the) Hodge bundle whose sections are differential forms on the fibers. Note that we underline the canonical line bundle $\underline{\omega}_\pi$ to distinguish it from the log-canonical line bundle $\omega_\pi$.
The perfect obstruction theory on $\cD^*$ relative to the stack of stable log-curves $\underline{\cM}$ is then given by a morphism
$$\phi \colon \left[\cO_{\cD^*} \to \mathbb{E}^\vee \right]^\vee \to L^\bullet_{\underline{\cD}^*/\underline{\cM}}.$$
Note that the stack of stable log-curves $\underline{\cM}$ is not the moduli space of stable curves, but we rather have
$$\cM := \mathcal{L}og^{\mathrm{fs}}_{(\overline{\cM}_{g,n},\cM_{\overline{\cM}_{g,n}}^\NN)}.$$
We have a forgetful map $\mathrm{Log} \colon \cM \to \mathcal{L}og^{\mathrm{fs}}_{(\mathrm{pt},\NN)}$ forgetting the family $C/S$ of stable curves and the morphism $\cM_S^{C/S} \oplus_{\cO^*_S} \NN \to \cM_S$, but remembering the log-structure $\cM_S$.
Since the stack $(\overline{\cM}_{g,n},\cM_{\overline{\cM}_{g,n}}^\NN)$ is log-smooth, then the map $\underline{\mathrm{Log}}$ is a smooth and representable morphism of algebraic stacks, see \cite[Section 2.1]{QCheneval}.
Therefore, the log-cotangent complex $L_{\underline{\mathrm{Log}}}$ is represented by a locally free sheaf, that we identify as the cotangent bundle of the moduli space of stable curves $\overline{\cM}_{g,n}$
$$T^*\overline{\cM}_{g,n} = \pi_* (\underline{\omega}_\pi \otimes \omega_\pi).$$
We have a distinguished triangle of cotangent complexes on $\cD^*$
$$q^*L_{\underline{\mathrm{Log}}} \to L_{\alpha} \to L_q \to q^*L_{\underline{\mathrm{Log}}}\left[1\right]$$
associated to the composition $\alpha$ of the morphisms
$$\underline{\cD}^* \xrightarrow{q} \underline{\cM} \xrightarrow{\underline{\mathrm{Log}}}
\underline{\mathcal{L}og}^{\mathrm{fs}}_{(\mathrm{pt},\NN)}.$$
Then we obtain a map
$$\left[\pi_* (\underline{\omega}_\pi) \to \cO_{\cD^*} \right] \to \left[ \pi_* (\underline{\omega}_\pi \otimes \omega_\pi) \to 0 \right],$$
and a perfect obstruction theory on the moduli space $\cD^*$ relative to the stack $\underline{\mathcal{L}og}^{\mathrm{fs}}_{(\mathrm{pt},\NN)}$ given by the cone of this map, i.e.~by the morphism
$$F^\bullet \to L_{\underline{\cD}^*/\underline{\mathcal{L}og}^{\mathrm{fs}}_{(\mathrm{pt},\NN)}} \quad \textrm{with} \quad F^\bullet:=\left[\pi_* (\underline{\omega}_\pi) \to \cO_{\cD^*} \oplus \pi_* (\underline{\omega}_\pi \otimes \omega_\pi) \right] \left[1\right].$$
As in the proof of \cite[Proposition 5.1.2]{QCheneval}, the complex $F^\bullet$ determines a vector bundle stack $\mathbb{F}$ over the moduli space $\underline{\cD}^*$, and the pair $(\alpha,\mathbb{F})$ satisfies condition $(\star)$ from \cite[Condition 3.3]{Manolache}. We use the virtual pull-back \cite[Definition 3.7]{Manolache}
$$\alpha_\mathbb{F}^! \colon A_*(\underline{\mathcal{L}og}^{\mathrm{fs}}_{(\mathrm{pt},\NN)}) \to A_{*-\delta}(\underline{\cD}^*),$$
where the integer $\delta$ is the virtual rank of the vector stack $\mathbb{F}$; we get
$$\delta = g-1-(3g-3+n)=-(2g-2+n).$$
The stack $\underline{\mathcal{L}og}^{\mathrm{fs}}_{\mathrm{pt}}$ is of pure dimension $0$ and is stratified by global quotients, see \cite[Definition 4.5.3]{Kresch}. The stack $\underline{\mathcal{L}og}^{\mathrm{fs}}_{(\mathrm{pt},\NN)}$ is one such strata, and is of pure dimension $-1$.
It has a fundamental class $\left[ \underline{\mathcal{L}og}^{\mathrm{fs}}_{(\mathrm{pt},\NN)} \right]$, and we have
$$\left[ \cD^* \right]^\mathrm{vir} = \alpha_\mathbb{F}^! \left[ \underline{\mathcal{L}og}^{\mathrm{fs}}_{(\mathrm{pt},\NN)} \right] \in A_{2g-3+n}(\cD^*).$$
Taking the sum of all the virtual dimensions in the decomposition \eqref{decom}, we obtain
$$\mathrm{vdim} = 3g-3+n - (g + \#E(\Gamma) - h_1(\Gamma) = 2g-2+n-\#V(\Gamma).$$
For the particular genus-$0$ case, observe that the Hodge bundle $\mathbb{E}$ vanishes, so that there are no obstructions.
\end{proof}

\subsection{Twisted canonical divisors}\label{twiscan}
According to \cite[Definition 20]{PandhFark}, a divisor $$\sum_{i=1}^n \mathbf{m}_i \sigma_i$$ associated to a marked stable curve $(C,\sigma_1,\dotsc,\sigma_n)$ is twisted canonical if there exists a twist $I$ for which
$$\nu^* \cO_C \left( \sum_{i=1}^n \mathbf{m}_i \sigma_i \right) \simeq \nu^*(\underline{\omega}_C^k) \otimes \cO_{C_I} \left( \sum_{q \in N_I} I(q') \cdot q' + I(q'') \cdot q'' \right)$$
on the partial normalization $C_I$, where the line bundle $\underline{\omega}_C$ is the canonical and not the log-canonical line bundle ; we recall below the notations and the definition of a twist.

Consider a marked curve $(C,\sigma_1,\dotsc,\sigma_n)$, with dual graph $G_C$. We denote by $V(G_C)$ and $H(G_C)$ the sets of vertices and half-edges of the dual graph.
Then a function
$$I \colon H(G_C) \to \ZZ$$
is called a twist if it satisfies the balancing, vanishing, sign, and transitivity conditions as follows.
\begin{itemize}
\item Balancing: for any edge $e=(h,h')$, we have $I(h)=-I(h')$.
\item Define $\sim$ the minimal equivalence relation on the set $V(G_C)$, for which two vertices separated by a half-edge on which the function $I$ vanishes are equivalent.
\item Vanishing: if two vertices $v, v'$ are equivalent, then for any edge $e=(h,h')$ between them, we have $I(h)=I(h')=0$.
\item For two vertices $v,v'$, we declare $v > v'$ if there is an edge $e=(h,h')$ with $h$ incident to v and $h'$ incident to $v'$ such that $I(h)>0$. The sign condition requires it is a well-defined partial order on the equivalence classes of $\sim$.
\item Construct a directed graph $G_I$ by taking equivalence classes of $\sim$ for the vertices and placing a directed edge $\overline{v} \to \overline{v}'$ between two equivalence classes $\overline{v}$ and $\overline{v}'$ if there exists two representants $v>v'$ of $\overline{v},\overline{v}'$ in $G_C$ with an edge between them. 
\item Transitivity: the directed graph $G_I$ has no directed loops.
\end{itemize}
The set $N_I$ is the set of edges of the graph $G_C$ on which the function $I$ is non-zero.
The curve $C_I$ is the partial normalization $\nu \colon C_I \to C$ defined by normalizing exactly the nodes in the set $N_I$.
We denote by $q'$ and $q''$ the two preimages of a normalized node $q \in N_I$, and we also identify them with the corresponding half-edges in $H(G_C)$.
Note that the set of vertices of the graph $G_I$ is in bijection with the set of connected components of the curve $C_I$, and also that the values of the twist at the markings of the curve $C$ are irrelevant.
At last, we define the subset
$$\widetilde{\mathcal{H}}^k_g(\vec{\mathbf{m}}) = \left\lbrace (C,\sigma_1,\dotsc,\sigma_n) \in \overline{\cM}_{g,n} | \sum_{i=1}^n \mathbf{m}_i \sigma_i \textrm{ is a $k$-twisted canonical divisor} \right\rbrace$$
of $\overline{\cM}_{g,n}$, for any vector $\vec{\mathbf{m}} \in \ZZ^n$ satisfying $\sum_{i=1}^n \mathbf{m}_i = k(2g-2)$.
By \cite[Theorem 21]{PandhFark}, it is a closed substack of $\overline{\cM}_{g,n}$ with irreducible components all of dimension at least $2g-3+n$.

\begin{pro}\label{compare}
Let $\vec{\mathbf{c}}=(\mathbf{c}_1, \dotsc, \mathbf{c}_n)$ be a partition of $k (2g-2+n)$, and $C$ be a genus-$g$ stable curve with markings $\sigma_1,\dotsc,\sigma_n$.
We define $\mathbf{m}_i:=\mathbf{c}_i-1$.
We have the following equivalence:
the divisor
$$\sum_{i=1}^n \mathbf{m}_i \sigma_i$$
is a $k$-twisted canonical divisor if and only if there exists a rubber $k$-log-canonical divisor $$\xi = (\pi \colon C^\mathrm{un} \rightarrow S, (\sigma_1,\dotsc,\sigma_n), \cM_S^{C^\mathrm{un}/S} \rightarrow \cM_S ~ ; ~~  C^\mathrm{un} \xrightarrow{f} P, f^*\cM^{0,\infty} \xrightarrow{f^\flat} \cM_{C^\mathrm{un}})$$ over $S = \mathrm{Spec}(\CC)$ of topological type $(g,n,\vec{\mathbf{c}})$, with underlying prestable curve $C^\mathrm{un}$ whose stabilization is the marked curve $C$.
\end{pro}

\begin{proof}
Take a rubber $k$-log-canonical divisor $\xi$ as in the statement and consider its marked graph $G_\xi$, i.e.~the dual graph of the underlying curve $C^\mathrm{un}$, together with a sign for each vertex and each half-edge, and a contact order for each half-edge.

By the $\CC^*$-admissibility condition, the sign of a vertex is determined by the other data of $G_\xi$. Indeed, it is $0$ if and only if the corresponding irreducible component is not stable, otherwise it is given by the sign of any incident half-edge.

We represent the marked graph $G_\xi$ as follows:
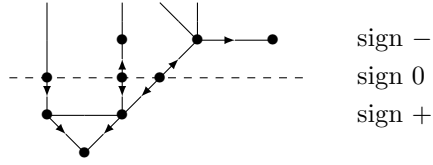
\begin{center}
\begin{tikzpicture}[scale=1]
\draw (0,0) node(A1){$\bullet$};
\draw (1,0) node(A2){$\bullet$};
\draw (0.5,-0.5) node(A3){$\bullet$};
\draw (1,1) node(A4){$\bullet$};
\draw (2,1) node(A5){$\bullet$};
\draw (3,1) node(A6){$\bullet$};

\draw (0,0)--(0.5,0);
\draw (0.5,0)--(1,0);

\draw (0,0)--(0,0.25);
\draw[<-,>=latex] (0,0.25)--(0,0.5);

\draw (0,0.5)--(0,1);
\draw (0,1)--(0,1.5);

\draw[->,>=latex] (0,0)--(0.25,-0.25);
\draw (0.25,-0.25)--(0.5,-0.5);

\draw[->,>=latex] (1,0)--(0.75,-0.25);
\draw (0.75,-0.25)--(0.5,-0.5);

\draw (1,0)--(1,0.25);
\draw[<-,>=latex] (1,0.25)--(1,0.5);

\draw[->,>=latex] (1,0.5)--(1,0.75);
\draw (1,0.75)--(1,1);

\draw (1,0)--(1.25,0.25);
\draw[<-,>=latex] (1.25,0.25)--(1.5,0.5);

\draw[->,>=latex] (1.5,0.5)--(1.75,0.75);
\draw (1.75,0.75)--(2,1);

\draw[->,>=latex] (2,1)-- (2.5,1);
\draw (2.5,1)-- (3,1);

\draw (1,1)-- (1,1.25);
\draw (1,1.25)-- (1,1.5);

\draw (2,1)-- (2,1.25);
\draw (2,1.25)-- (2,1.5);

\draw (2,1)-- (1.75,1.25);
\draw (1.75,1.25)-- (1.5,1.5);

\draw[dashed] (-0.5,0.5)--(3.5,0.5);

\draw (0,0.5) node{$\bullet$};
\draw (1,0.5) node{$\bullet$};
\draw (1.5,0.5) node{$\bullet$};

\draw (4,0) node[right]{sign $+$};
\draw (4,0.5) node[right]{sign $0$};
\draw (4,1) node[right]{sign $-$};
\end{tikzpicture}
\end{center}
where the arrows correspond to the orientation defined in \cite[Definition 3.3.2]{QChen}:
if an edge $(h,h')$ is such that $c_h \geq 0$, then we write $h \geq h'$. Note that edges with trivial contact orders have two directions (and are non-oriented in the example). It defines a partial order on the vertices as well.
Observe that an unstable vertex is always strictly bigger than its neighbours.

A cycle in the marked graph $G_\xi$ is a sequence of consecutive edges $(e_1,\dotsc,e_l,e_1)$ in the graph satisfying $e_1 \geq \cdot \geq e_l \geq e_1$.
A strict cycle is a cycle of $G_\xi$ containing at least one strict inequality.
By \cite[Proposition 3.4.3]{QChen} and \cite[Corollary 3.3.7]{QChen}, there is no strict cycle in the marked graph $G_\xi$.

As before, we write $\mathbf{c}_h = \mathbf{s}_h \cdot c_h$ for each half-edge.
Obviously, if two half-edges $h,h'$ form an edge, then we have $\mathbf{c}_{h} = - \mathbf{c}_{h'}$.
The orientation of $G_\xi$ given by the integers $\mathbf{c}_h$ instead of $c_h$ corresponds to reversing the arrows in the negative part. In our example, it gives:
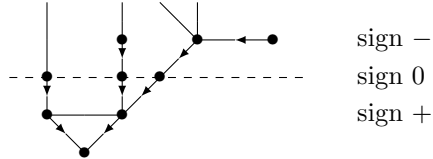
\begin{center}
\begin{tikzpicture}[scale=1]
\draw (0,0) node(A1){$\bullet$};
\draw (1,0) node(A2){$\bullet$};
\draw (0.5,-0.5) node(A3){$\bullet$};
\draw (1,1) node(A4){$\bullet$};
\draw (2,1) node(A5){$\bullet$};
\draw (3,1) node(A6){$\bullet$};

\draw (0,0)--(0.5,0);
\draw (0.5,0)--(1,0);

\draw (0,0)--(0,0.25);
\draw[<-,>=latex] (0,0.25)--(0,0.5);

\draw (0,0.5)--(0,1);
\draw (0,1)--(0,1.5);

\draw[->,>=latex] (0,0)--(0.25,-0.25);
\draw (0.25,-0.25)--(0.5,-0.5);

\draw[->,>=latex] (1,0)--(0.75,-0.25);
\draw (0.75,-0.25)--(0.5,-0.5);

\draw (1,0)--(1,0.25);
\draw[<-,>=latex] (1,0.25)--(1,0.5);

\draw (1,0.5)--(1,0.75);
\draw[<-,>=latex] (1,0.75)--(1,1);

\draw (1,0)--(1.25,0.25);
\draw[<-,>=latex] (1.25,0.25)--(1.5,0.5);

\draw (1.5,0.5)--(1.75,0.75);
\draw[<-,>=latex] (1.75,0.75)--(2,1);

\draw (2,1)-- (2.5,1);
\draw[<-,>=latex] (2.5,1)-- (3,1);

\draw (1,1)-- (1,1.25);
\draw (1,1.25)-- (1,1.5);

\draw (2,1)-- (2,1.25);
\draw (2,1.25)-- (2,1.5);

\draw (2,1)-- (1.75,1.25);
\draw (1.75,1.25)-- (1.5,1.5);

\draw[dashed] (-0.5,0.5)--(3.5,0.5);

\draw (0,0.5) node{$\bullet$};
\draw (1,0.5) node{$\bullet$};
\draw (1.5,0.5) node{$\bullet$};

\draw (4,0) node[right]{sign $+$};
\draw (4,0.5) node[right]{sign $0$};
\draw (4,1) node[right]{sign $-$};
\end{tikzpicture}
\end{center}
Moreover, an unstable vertex corresponds to $(\PP^1,0,\infty)$, so that it has exactly two incident half-edges $h$ and $h'$, and by the compatibility condition \eqref{compat}, we have $c_{h}=c_{h'}$. Hence, we have $\mathbf{c}_h=-\mathbf{c}_{h'}$.
As a consequence, we can stabilize the dual graph $G_\xi$ to $G_C$ and we obtain an induced orientation and an induced function
$$\mathbf{c} \colon H(G_C) \to \ZZ$$
satisfying the balancing condition.
On our example, we get:
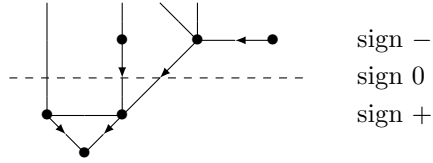
\begin{center}
\begin{tikzpicture}[scale=1]
\draw (0,0) node(A1){$\bullet$};
\draw (1,0) node(A2){$\bullet$};
\draw (0.5,-0.5) node(A3){$\bullet$};
\draw (1,1) node(A4){$\bullet$};
\draw (2,1) node(A5){$\bullet$};
\draw (3,1) node(A6){$\bullet$};

\draw (0,0)--(0.5,0);
\draw (0.5,0)--(1,0);

\draw (0,0)--(0,0.25);
\draw (0,0.25)--(0,0.5);

\draw (0,0.5)--(0,1);
\draw (0,1)--(0,1.5);

\draw[->,>=latex] (0,0)--(0.25,-0.25);
\draw (0.25,-0.25)--(0.5,-0.5);

\draw[->,>=latex] (1,0)--(0.75,-0.25);
\draw (0.75,-0.25)--(0.5,-0.5);

\draw (1,0)--(1,0.25);
\draw (1,0.25)--(1,0.5);

\draw[<-,>=latex] (1,0.5)--(1,0.75);
\draw (1,0.75)--(1,1);

\draw (1,0)--(1.25,0.25);
\draw (1.25,0.25)--(1.5,0.5);

\draw[<-,>=latex] (1.5,0.5)--(1.75,0.75);
\draw (1.75,0.75)--(2,1);

\draw (2,1)-- (2.5,1);
\draw[<-,>=latex] (2.5,1)-- (3,1);

\draw (1,1)-- (1,1.25);
\draw (1,1.25)-- (1,1.5);

\draw (2,1)-- (2,1.25);
\draw (2,1.25)-- (2,1.5);

\draw (2,1)-- (1.75,1.25);
\draw (1.75,1.25)-- (1.5,1.5);

\draw[dashed] (-0.5,0.5)--(3.5,0.5);


\draw (4,0) node[right]{sign $+$};
\draw (4,0.5) node[right]{sign $0$};
\draw (4,1) node[right]{sign $-$};
\end{tikzpicture}
\end{center}
Observe that we have only changed the orientation in the negative part and that an edge from the negative part to the positive part is always oriented towards the bottom.
Therefore, since the original marked graph $G_\xi$ has no strict cycles, we deduce that the directed graph $G_C$ has no strict cycles.
As a consequence, the function $\mathbf{c} \colon H(G_C) \to \ZZ$ has to satisfy the vanishing, the sign, and the transitivity conditions.
Then it is a twist on the stable curve $C$.
Obviously, the function $-\mathbf{c}$ is also a twist on the stable curve $C$.

Furthermore, by Proposition \ref{detail}, we have a non-vanishing global section of
$$s \in H^0(\widetilde{C}^\mathrm{un},\omega^k_{\widetilde{C}^\mathrm{un}}(D_\infty - D_0))$$
representing the $k$-log-canonical divisor $\xi$, where the curve $\widetilde{C}^\mathrm{un}$ is the partial normalization of the prestable curve $C^\mathrm{un}$ at the nodes with non-trivial contact orders, and where $\omega$ denotes the log-canonical line bundle.
The section $s$ yields an isomorphism
\begin{equation}\label{is}
\omega^k_{\widetilde{C}^\mathrm{un}} \simeq \cO_{\widetilde{C}^\mathrm{un}}(\sum_x \mathbf{c}_x \cdot x)
\end{equation}
where the sum is taken over the markings of the curve $\widetilde {C}^\mathrm{un}$.
To conclude, observe the relation
$$\widetilde{C}^\mathrm{un} = C_I \sqcup \bigsqcup (\PP^1,0,\infty),$$
where the disjoint union is taken over the set of unstable irreducible components of the curve $C^\mathrm{un}$, and that the equation \eqref{is} is trivial on $(\PP^1,0,\infty)$.
Therefore, the divisor $\sum_{i=1}^n \mathbf{m}_i \sigma_i$
is a $k$-twisted canonical divisor.

Conversely, take a stable marked curve $(C,\sigma_1,\dotsc,\sigma_n)$ such that the divisor $\sum_{i=1}^n \mathbf{m}_i \sigma_i$ is a $k$-twisted canonical divisor.
Therefore, there exists a twist $I$ on the dual graph $G_C$. We define the function $\mathbf{c} \colon H(G_C) \to \ZZ$ by taking $-I$ and then replacing the values at the markings by $\mathbf{c}_i:=\mathbf{m}_i+1$.
The dual graph $G_C$ has an orientation given by the function $\mathbf{c}$.

We consider a partition of the set $H(G_C)$ of half-edges of the graph $G_C$ into two parts
$$H(G_C) = H^+(G_C) \sqcup H^-(G_C),$$
such that
\begin{itemize}
\item if $h$ is a leg with $\mathbf{c}_h>0$, then $h \in H^+(G_C)$,
\item if $h$ is a leg with $\mathbf{c}_h<0$, then $h \in H^-(G_C)$,
\item if two half-edges $h,h'$ form a (strict) oriented edge $h \to h'$, then we do not have $(h,h') \in H^+(G_C) \times H^-(G_C)$,
\item all the half-edges that are not legs and that are incident to the same vertex are in the same subset.
\end{itemize}
If $C$ is not a smooth curve, then this dichotomy among half-edges induces a dichotomy among vertices.
If $C$ is a smooth curve, then we choose a sign for the vertex.

In the light of our example above, we have made the choice of a dashed line cutting the dual graph $G_C$ in two parts, that we call positive and negative.
Then, we insert a vertex with genus $0$ at each intersection point of the dashed line with $G_C$, or equivalently, we insert the dual graph of $(\PP^1,0,\infty)$ between at the middle of every edge with half-edges of different signs, and also between a vertex and a leg if the leg is attached to the vertex but has different sign.
At last, we define $c_h := \mathbf{s}_h \cdot \mathbf{c}_h$, where $\mathbf{s}_h$ stands for the sign of the half-edge $h$.
Therefore, we obtain a marked graph, admissible in the sense of \cite[Definition 3.3.6]{QChen}.
Moreover, since we started with a $k$-twisted canonical divisor, the equation \eqref{is} is satisfied.
Thus, by Proposition \ref{detail}, we have a unique rubber $k$-log-canonical divisor $\xi$ with the marked graph we have constructed.
\end{proof}

In the proof of Proposition \ref{compare}, we have seen that there are several rubber $k$-log-canonical divisors associated to the same $k$-twisted canonical divisor, due to the choice of a twist function on the dual graph and to the choice of a dichotomy between half-edges.
Moreover, we have seen in Definition \ref{topoltwist} that the moduli space of rubber $k$-log-canonical divisors of type $(g,n,\vec{\mathbf{c}})$ is not connected, and that we have an extra topological data given by a decorated dual graph.
We fix this issue by considering the following definition.

\begin{dfn}
A rubber $k$-log-canonical divisor $\xi$ is called principal if the schematic map is the zero section on every stable component of the source curve.
We denote by $\cD^{*(k)}_{\mathrm{pr},(g,n,\vec{\mathbf{c}})} \subset \cD_{(g,n,\vec{\mathbf{c}})}^{*(k)}$ the open and closed substack corresponding to principal rubber $k$-log-canonical divisors.
For principal rubber $k$-log-canonical divisors with at least one negative integer $\mathbf{c}_i$, we say that we are in the strictly meromorphic case.
For $k=1$, it coincides with the definition in \cite{PandhFark}.
\end{dfn}

\begin{rem}
Let $\xi$ be a rubber $k$-log-canonical divisor with source curve $C$.
It is principal if and only if every unstable component $(\PP^1,0,\infty) \subset C$ carries a marking with $\mathbf{c}_i<0$.
Following Definition \ref{topoltwist}, the substack $\cD^{*(k)}_{\mathrm{pr},(g,n,\vec{\mathbf{c}})} \subset \cD_{(g,n,\vec{\mathbf{c}})}^{*(k)}$ consists of points of $\cD^{*(k)}$ where the topological type is $(g,n,\vec{\mathbf{c}})$ enhanced with the data of the trivial dual graph (with no edges) where the negative markings are twisted by their contact orders.
We have a similar component of the moduli space $\cD^{*(k)}$ where the schematic map is the infinity section on every stable component.
\end{rem}

\begin{cor}
Let $\vec{\mathbf{c}}=(\mathbf{c}_1, \dotsc, \mathbf{c}_n)$ be a partition of $k (2g-2+n)$, and define $\mathbf{m}_i:=\mathbf{c}_i-1$.
Then, the natural map
$$\mathrm{st} \colon \cD^{*(k)}_{\mathrm{pr},(g,n,\vec{\mathbf{c}})} \to \overline{\cM}_{g,n},$$
is a finite morphism of Deligne--Mumford stacks, whose image is the closed substack $\widetilde{\mathcal{H}}^k_g(\vec{\mathbf{m}}) \subset \overline{\cM}_{g,n}$.
Therefore, the dimension of the moduli space $\cD^{*(k)}_{\mathrm{pr},(g,n,\vec{\mathbf{c}})}$ is greater than its virtual dimension $2g-3+n$.

In the strictly meromorphic case of principal rubber $1$-log-canonical divisors, the moduli space $\cD^{*(1)}_{\mathrm{pr},(g,n,\vec{\mathbf{c}})}$ is of the expected virtual dimension $2g-3+n$, and the virtual class equals the fundamental class.
\end{cor}

%
%

\begin{proof}
We do the proof in the more general case of the space $\cD^{* (k)}_\Gamma$, where $\Gamma$ is the topological type $(g,n,\vec{\mathbf{c}})$ enhanced with a decorated dual graph as in Definition \ref{topoltwist}.
By Proposition \ref{compare}, the geometric points of $\cD^{*(k)}_\Gamma$ are sent to geometric points of $\widetilde{\mathcal{H}}^k_g(\vec{\mathbf{m}})$ by the morphism $\mathrm{st}$.
This morphism is the composition
$$\cD^{*(k)}_\Gamma \to D^{*(k)}_\Gamma \hookrightarrow \overline{\cM}_{g,n}(\PP^1) \to \overline{\cM}_{g,n}$$
and the first morphism is finite by Corollary \ref{DMstar}.
The image of $D^{*(k)}_\Gamma$ into $\overline{\cM}_{g,n}(\PP^1)$ consists of stable maps to $\PP^1$ such that each stable component maps to $0$ or $\infty$, so that the morphism forgetting the stable map and stabilizing the source curve is also finite.
The map $\mathrm{st}$ is then finite, and the dimension of $\cD^{*(k)}_\Gamma$ equals the dimension of its image, which is a closed substack of $\overline{\cM}_{g,n}$, and thus a closed substack of $\widetilde{\mathcal{H}}^k_g(\vec{\mathbf{m}})$ since it is itself closed.
By the decomposition \eqref{decom}, we may assume the stabilized graph of $\Gamma$ has no edges, i.e.~we are in the principal case.
Then the map $\mathrm{st}$ is surjective onto 
$\widetilde{\mathcal{H}}^k_g(\vec{\mathbf{m}})$, and by \cite[Theorem 21]{PandhFark}, all the irreducible components of $\widetilde{\mathcal{H}}^k_g(\vec{\mathbf{m}})$ are of dimensions at least $2g-3+n$, which is the virtual dimension in the principal case.
For the special case $k=1$, in the strictly meromorphic case, the dimension of $\widetilde{\mathcal{H}}^1_g(\vec{\mathbf{m}})$ is exactly $2g-3+n$, by \cite[Theorem 3]{PandhFark}.
\end{proof}

\subsection{Multiplicities of the moduli space}
The moduli space $\cD^{*(k)}$ of rubber $k$-log-canonical divisors is in general not reduced, due to the presence of non-zero contact orders at the nodes.
Furthermore, in the strictly meromorphic case for principal rubber $1$-log-canonical divisors, the virtual class equals the fundamental class
$$\left[ \cD^{*(1)}_{\mathrm{pr},(g,n,\vec{\mathbf{c}})} \right]^\mathrm{vir} = \left[ \cD^{*(1)}_{\mathrm{pr},(g,n,\vec{\mathbf{c}})} \right].$$
In the following proposition, we compute the multiplicities of the irreducible components of the moduli space $\cD^{*(k)}$, and then relate them to a conjecture in \cite[Appendix]{PandhFark}.

\begin{pro}\label{multipli}
Let $\xi_0$ be a rubber $k$-log-canonical divisor with marked graph $G_0$, and denote by $G_0^\mathrm{stab}$ the stabilized dual graph with induced contact orders. The sets $E^*(G_0)$ and $E^*(G_0^\mathrm{stab})$ stand for the sets of edges with non-zero contact orders.
Let $Z_0$ be the strata of $\cD^{*(k)}$ where the generic point has marked graph $G_0$.
Then the multiplicity of the strata $Z_0$ is given by
$$\prod_{e \in E^*(G_0)} |c_e|,$$
where $\pm c_e$ is the non-zero contact order at the edge $e$.
Furthermore, the push-forward of the fundamental class of the component $Z_0$ to the moduli space of stable curves is
$$\mathrm{st}_* \left[Z_0\right] = 
\prod_{e \in E^*(G_0^\mathrm{stab})} |c_e| \cdot \left[\mathrm{st}(Z_0) \right] \in A^*(\overline{\cM}_{g,n}).$$
\end{pro}

\begin{proof}
Forgetting the log-map $(f,f^\flat)$, we have a morphism sending a rubber $k$-log-canonical divisor 
$$\xi = (C \rightarrow S, (\sigma_1,\dotsc,\sigma_n), \cM_S^{C/S} \rightarrow \cM_S ~ ; ~~  C \xrightarrow{f} P, f^*\cM^{0,\infty} \xrightarrow{f^\flat} \cM_C)$$
to the prestable log-curve
$$(C \rightarrow S, (\sigma_1,\dotsc,\sigma_n), \cM_S^{C/S} \rightarrow \cM_S).$$
This morphism is represented by a reduced scheme, and the morphism of log-structures $\cM_S^{C/S} \rightarrow \cM_S$ is given by multiplying the elements smoothing the nodes by the contact orders at the nodes. Thus, it is obtained by a root construction, as described in \cite[Appendix A.3.3]{QChendegen}, over the moduli space of prestable curves with standard log-structures, which is isomorphic to $\mathfrak{M}_{g,n}$ and thus reduced.
By \cite[Theorem 3.3.6]{Cad07} and \cite[Section 4]{MO05}, the root construction gives a non-reduced stack, and the multiplicity is the product of the contact orders at the nodes.

For the second part of the statement, observe that the morphism $\mathrm{st}$ is in general not representable, since we have an extra automorphism group $\mu_{c_v}$ for each unstable component $v \in V^\mathrm{un}(G_0)$, where $c_v$ denotes the contact order of the component $v$ and $V^\mathrm{un}(G_0)$ stands for the set of unstable vertices of the dual graph $G_0$.
Then we obtain the formula
$$\mathrm{st}_* \left[Z_0\right] = 
\cfrac{\prod_{e \in E^*(G_0)} |c_e| }{\prod_{v \in V^\mathrm{un}(G_0)} |c_v|} \cdot \left[\mathrm{st}(Z_0) \right] \in A^*(\overline{\cM}_{g,n}).$$
An unstable component $v \in V^\mathrm{un}(G_0)$ is of type $(\PP^1,0,\infty)$, and we have two cases:
\begin{itemize}
\item the two special points of $v$ are nodes with contact orders $|c_v|$,
\item exactly one special point of $v$ is a node with contact order $|c_v|$; the other point is a marking.
\end{itemize}
In the first case, the unstable component $v$ contributes the factor $|c_v|$ to the pushforward formula.
In the second case, it contributes trivially to the formula.
\end{proof}

\begin{cor}\label{weighted}
Take $k=1$ and fix a triplet $(g,n,\vec{\mathbf{c}})$; we write $\mathbf{m}_i:=\mathbf{c}_i-1$.
Then we have
$$\mathrm{st}_* \left[ \cD^{*(1)}_{\mathrm{pr},(g,n,\vec{\mathbf{c}})} \right] = H_{g,\vec{\mathbf{m}}} \in A^g(\overline{\cM}_{g,n}),$$ where the right-hand side is the weighted fundamental class from \cite[Appendix A.4]{PandhFark}.
\end{cor}

\begin{conj}[{\cite[Conjecture A]{PandhFark}}]
In the strictly meromorphic case, the moduli space of principal rubber $1$-log-canonical divisors satisfies
$$\mathrm{st}_* \left[ \cD^{*(1)}_{\mathrm{pr},(g,n,\vec{\mathbf{c}})} \right] = 2^{-g} P^g_{g,\vec{\mathbf{m}}}.$$
\end{conj}

\begin{conj}
For any integer $k$ and any triplet $(g,n,\vec{\mathbf{c}})$, we have
$$\mathrm{st}_* \left[ \cD^{*(k)}_{\mathrm{pr},(g,n,\vec{\mathbf{c}})} \right]^\mathrm{vir} = 2^{-g} P^{g,k}_{g,\vec{\mathbf{m}}},$$
where $P^{g,k}_{g,\vec{\mathbf{m}}}$ is Pixton's cycle.
\end{conj}

\end{document}